\documentclass[]{amsart}
\usepackage{amssymb}
\usepackage{hyperref}
\usepackage{url}

\title{A Strong Duality Principle for~\\Equivalence Couplings and Total Variation} % \thanks is optional. Insert line breaks with \\

\subjclass[2020]{60A10; 03E15; 28A35; 90C46.}

\keywords{coupling; equivalence relation; total variation; duality; optimal transport; Kantorovich duality; Borel equivalence relation; smoothness; hypersmoothness}

\author{Adam~Quinn~Jaffe}
\address{University of California, Berkeley, United States of America}
\email{aqjaffe@berkeley.edu} %AUTHORS

\newcommand{\R}{\mathbb{R}}
\newcommand{\C}{\mathbb{C}}
\newcommand{\N}{\mathbb{N}}
\newcommand{\Q}{\mathbb{Q}}
\newcommand{\M}{\mathbb{M}}
\newcommand{\W}{\mathbb{W}}

\DeclareMathOperator{\atom}{at}

\DeclareMathOperator{\TV}{TV}

\DeclareSymbolFont{bbold}{U}{bbold}{m}{n}
\DeclareSymbolFontAlphabet{\mathbbold}{bbold}
\newcommand{\ind}{\mathbbold{1}}

\renewcommand{\P}{\mathbb{P}}
\newcommand{\E}{\mathbb{E}}
\newcommand{\F}{\mathcal{F}}
\newcommand{\G}{\mathcal{G}}

\renewcommand{\S}{\mathcal{S}}
\newcommand{\tail}{\mathcal{T}}
\newcommand{\invar}{\mathcal{I}}
\newcommand{\exch}{\mathcal{E}}

\renewcommand{\emptyset}{\varnothing}

\theoremstyle{plain}
\newtheorem{theorem}{Theorem}[section]

\newtheorem{conjecture}[theorem]{Conjecture}
\newtheorem{corollary}[theorem]{Corollary}

\newtheorem{definition}[theorem]{Definition}

\newtheorem{lemma}[theorem]{Lemma}

\newtheorem{proposition}[theorem]{Proposition}

\newtheorem{innercustomgeneric}{\customgenericname}
\providecommand{\customgenericname}{}
\newcommand{\newcustomtheorem}[2]{%
	\newenvironment{#1}[1]
	{%
		\renewcommand\customgenericname{#2}%
		\renewcommand\theinnercustomgeneric{##1}%
		\innercustomgeneric
	}
	{\endinnercustomgeneric}
}

\newcustomtheorem{customthm}{Theorem}
\newcustomtheorem{customcor}{Corollary}

\theoremstyle{remark}
\newtheorem{example}[theorem]{Example}

\newtheorem{remark}[theorem]{Remark}

\numberwithin{equation}{section}

\begin{document}
		
		\maketitle
		%\date{}
		
		\begin{abstract}
			We introduce and study a notion of duality for two classes of optimization problems commonly occurring in probability theory.
			That is, on an abstract measurable space $(\Omega,\F)$, we consider pairs $(E,\G)$ where $E$ is an equivalence relation on $\Omega$ and $\G$ is a sub-$\sigma$-algebra of $\F$; we say that $(E,\G)$ satisfies ``strong duality'' if $E$ is $(\F\otimes\F)$-measurable and if for all probability measures $\P,\P'$ on $(\Omega,\F)$ we have
			\begin{equation*}
				\max_{A\in\G}\vert \P(A)-\P'(A)\vert = \min_{\tilde{\P}\in\Pi(\P,\P')}(1-\tilde{\P}(E)),
			\end{equation*}
			where $\Pi(\P,\P')$ denotes the space of couplings of $\P$ and $\P'$, and where ``max'' and ``min'' assert that the supremum and infimum are in fact achieved.
			The results herein give wide sufficient conditions for strong duality to hold, thereby extending a form of Kantorovich duality to a class of cost functions which are irregular from the point of view of topology but regular from the point of view of descriptive set theory.
			The given conditions recover or strengthen classical results, and they have novel consequences in stochastic calculus, point process theory, and random sequence simulation.
		\end{abstract}

\section{Introduction}

The objects of interest in this paper are two optimization problems commonly occurring in probability theory.
To state them, we consider two probability measures $\P$ and $\P'$ on an abstract measurable space $(\Omega,\F)$.
For an equivalence relation $E$ on $\Omega$ (which, when viewed as a subset of $\Omega\times\Omega$, is $(\F\otimes\F)$-measurable) we often aim to solve
\begin{equation}\label{eqn:primal}
	\begin{split}
		\text{minimize} & \qquad 1-\tilde{\P}(E) \\
		\text{over} &\qquad \text{all couplings } \tilde{\P} \text{ of } \P,\P',
	\end{split} \tag{EC}
\end{equation}
which we refer to as the \textit{equivalence coupling problem for $E$} or simply the \textit{$E$-coupling problem}.
Alternatively, for a sub-$\sigma$-algebra $\G$ of $\F$, we might aim to solve
\begin{equation}\label{eqn:dual}
	\begin{split}
		\text{maximize} & \qquad |\P(A) - \P'(A)| \\
		\text{over} &\qquad A\in \G,
	\end{split} \tag{TV}
\end{equation}
which we refer to as \textit{total variation problem for $\G$} or simply the \textit{$\G$-total variation problem}.
While one is typically interested in one of \eqref{eqn:primal} or \eqref{eqn:dual} for particular applications, our goal in this paper is to show that, in great generality, they are in fact equivalent.

\subsection{Review of Prior Work}\label{subsec:prior}

There already exist a few particular instances of strong duality which motivate the potential success of such a general theory; the most classical example will certainly be familiar to the reader.
It states that, if $\Omega$ is Polish space with $\mathcal{B}(\Omega)$ is its Borel $\sigma$-algebra, and if $\Delta = \{(\omega,\omega)\in \Omega\times \Omega: \omega\in \Omega \}$ denotes the diagonal in $\Omega\times\Omega$, we have
\begin{equation}\label{eqn:simple-duality}
	\max_{A\in\mathcal{B}(\Omega)}\vert \P(A)-\P'(A)\vert = \min_{\tilde{\P}\in\Pi(\P,\P')}(1-\tilde{\P}(\Delta))
\end{equation}
for all Borel probability measures $\P,\P'$ on $\Omega$, where $\Pi(\P,\P')$ denotes the space of all couplings of $\P$ and $\P'$.
This result has certainly been known for a long time, at least for countable sets $\Omega$, so its exact source is difficult to track down \cite[Chapter~I.7]{LindvallCoupling}.

A second known example concerns the space of binary sequences $\Omega := \{0,1\}^{\N}$ with the Borel $\sigma$-algebra of its product topology, and $E_0$ the equivalence relation of eventual equality.
(Actually, $\{0,1\}$ can be replaced with any Polish space.)
A coupling of two Borel probability measures $\P,\P'$ on $\Omega$ is called \textit{successful} if $E_0$ has probability one, that is, if the two random sequences are eventually equal almost surely.
It was shown in a series of works \cite{Pitman, Follmer, Thorisson, Griffeath}, culminating in \cite{Goldstein}, that the existence of a successful coupling of two Borel probability measures $\P,\P'$ on $\Omega$ is closely related to $\P$ and $\P'$ assigning the same probability to all elements of the tail $\sigma$-algebra $\tail$.
More precisely \cite[Theorem~2.1]{Goldstein}, one has
\begin{equation}\label{eqn:tail-duality}
	\max_{A\in \tail}|\P(A) - \P'(A)| = 0 \qquad \text{if and only if} \qquad \min_{\tilde{\P}\in\Pi(\P,\P')}(1-\tilde{\P}(E_0)) = 0,
\end{equation}
for all Borel probability measures $\P,\P'$ on $\Omega$.

A third example, from ergodic theory, shows that eventual equality and the tail $\sigma$-algebra in \eqref{eqn:tail-duality} can be replaced with the analogous objects for the notion of shift-invariance.
Adopting the notation of the preceding paragraph, write $\invar$ for the shift-invariant $\sigma$-algebra and $E_{S}$ for the equivalence relation of equality modulo shifts.
A coupling of two probability measures $\P,\P'$ on $\Omega$ is called a \textit{successful shift-coupling} if $E_S$ has probability one, that is, if the random sequences are almost surely equal modulo a random shift.
Then, it was shown \cite{AldousThorisson} that one has
\begin{equation}\label{eqn:shift-duality}
	\max_{A\in \invar}|\P(A) - \P'(A)| = 0 \qquad \text{if and only if} \qquad \min_{\tilde{\P}\in\Pi(\P,\P')}(1-\tilde{\P}(E_S)) = 0,
\end{equation}
for $\P,\P'$ any two Borel probability measures on $\Omega$.

Subsequent work also saw many generalizations of \eqref{eqn:shift-duality}.
For example, in \cite{Georgii} it is shown that a similar statement holds for measurable actions of very many groups and semigroups on Polish spaces.
There is also \cite{Shinko} which generalizes this fact to the setting of cardinal algebras.
We also have \cite[Theorem~1]{Khezeli} which establishes an analogous statement for unimodular random networks, where the shift operation is replaced by a root-change operation; this example is particularly interesting since the underlying measurable space is not standard Borel \cite[Proposition~1]{Khezeli}.

The most general formulation of an existing result in the spirit of strong duality that we are aware of is \cite[Theorem~1]{PratelliRigo} which proves, as a consequence of a much more general result on constrained versions of the Skorokhod representation theorem, the following:
If two Borel probability measures on a metric space $S$ agree on the sub-$\sigma$-algebra generated by a measurable function $g:S\to T$ for a sufficiently nice space $T$, then they can be coupled together so that their images under $g$ are almost surely equal.
This is a sort of strong duality statement for equivalence relations on $S$ that can be represented as pullbacks of $\Delta$ in $T$ by $g$; however, the regularity conditions required on $T$ mean this result does not recover strong duality for many examples of interest (for example, the pairs $(E_0,\tail)$ and $(E_{S},\invar)$ above).

\subsection{Relation to Kantorovich Duality}\label{subsec:Kantorovich}

The theory of Monge-Kantorovich optimal transport and Kantorovich duality can also be seen as an important precedent for our work; see \cite{OldAndNew,VillaniTopics} for the standard background, and \cite{RachevRdorf} for a more comprehensive treatment which is better suited to our setting.
Since \eqref{eqn:primal} is exactly a Monge-Kantorovich optimal transport problem with cost $c:=1-\ind_E$, it is natural to try to prove strong duality as a consequence of some sufficiently general theorem in the vast literature on Kantorovich duality \cite{OldAndNew,VillaniTopics,Kellerer,RamachandranRdorf,Rigo,Bieglbock1,Bieglbock2}.
In this subsection we address the limitations of such an approach.

First, we consider the case that $\Omega$ is Polish and $E$ is closed in $\Omega\times\Omega$.
Then the cost function $c=1-\ind_E$ is lower-semicontinuous, so a standard form of Kantorovich duality \cite[Theorem~5.10]{OldAndNew} gives
\begin{equation*}
	\max_{f,f'}\left(\int_{\Omega}f\,d\P+\int_{\Omega}f'\,d\P'\right)= \min_{\tilde{\P}\in\Pi(\P,\P')}(1-\tilde{\P}(E)),
\end{equation*}
where the right side is achieved and the left side is taken over, say, all bounded measurable $f,f':\Omega\to\R$ with $f\oplus f' \le 1-\ind_E$.
In fact, it is standard that one can equivalently take the maximum over function classes which are assumed to have further structure, hence massaging the left side in the following ways:
first, one can assume that $f$ is $\P$-integrable and $c$-convex;
second, one can assume that $f'$ is equal to $f^c$, the $c$-transform of $f$;
third, one can use that $c$ satisfies the triangle inequality (this follows from $E$ being transitive) to get that $f^c = -f$;
fourth, one can use that $c$ satisfies the bounds $0\le c\le 1$ to assume that $f$ satisfies the bounds $0\le f \le 1$;
fifth, one can use that $c$ is symmetric (this follows from $E$ being reflexive) to replace the parentheses with absolute values.
At this point some straightforward additional analysis shows that one can take $f$ to be of the form $\ind_A$ for $A$ ranging over a suitable sub-$\sigma$-algebra $\G$ of $\F$.
(Alternatively, one can begin by directly appling a form of Kantorovich duality which is specifically tailored to cost functions which only take on values in $\{0,1\}$, like \cite[Theorem~1.27]{VillaniTopics}.)
Thus, Kantorovich duality implies that closed equivalence relations satisfy our strong duality.

However, for the applications most of interest in probability theory, one needs to consider the case that $\Omega$ is Polish and $E$ is $F_{\sigma}$ in $\Omega\times\Omega$.
Since $c=1-\ind_E$ is no longer guaranteed to be lower semicontinuous, the standard form of Kantorovich duality does not apply.
Nonetheless, one has \cite[Corollary~2.3.9]{RachevRdorf} for all measurable costs, which guarantees
\begin{equation*}
	\sup_{f,f'}\left(\int_{\Omega}f\,d\P+\int_{\Omega}f'\,d\P'\right)= \inf_{\tilde{\P}\in\Pi(\P,\P')}(1-\tilde{\P}(E)),
\end{equation*}
where again the left side is taken over all bounded measurable $f,f':\Omega\to\R$ with $f\oplus f' \le 1-\ind_E$.
While this may give the impression that strong duality is within reach, there are two crucial ways in which we become stuck:
For one, there are not sufficient results in the literature which guarantee that the left side can be massaged, as in the preceding paragraph, into the form of a total variation norm; indeed, steps three, four, and five can be executed similarly, but steps one and two are hard to rigorosouly justify, due to the fact that classical approximation arguments break down when $c$ lacks the appropriate topological regularity.
For another, which we explain further in the next paragraph, the existence of primal minimizers is not guaranteed.
Thus, existing forms of Kantorovich duality do not appear to imply that $F_{\sigma}$ equivalence relations satisfy our strong duality for a suitable sub-$\sigma$-algebra $\G$ of $\F$.

In our opinion, the question of the existence of primal minimizers has received surprisingly little attention.
In the classical setting where $\Omega$ is a Polish space and the cost $c$ is lower semi-continuous, existence of course follows immediately from topological considerations.
Yet, many authors do not attempt to expand the scope of existence results; for instance, in \cite{Bieglbock2} it is stated: ``If $c$ fails to be lower semi-continuous, there is little reason why a primal optimizer should exist''.
Nonetheless, there are, as far as we know, exactly two known results outside the classical setting:
The first is \cite[Remark~2.12(a)]{RachevRdorf} in which it is shown that primal minimizers always exist if one works in the world of finitely-additive probability measures, but this is not useful to us since we are interested in the more standard setting of countably-additive probability measures.
The second is \cite[Theorem~2.3.10]{RachevRdorf} (originally proven in \cite{Kellerer}) in which it is shown that primal minimizers exist for costs which are limits of regular costs, with respect to a certain metric; unfortunately, this condition is still not general enough to cover our setting of interest.

For all these reasons, the state-of-the-art technology on Kantorovich duality does not appear to imply our notion of strong duality in a sufficiently general setting.
Nonetheless, our strong duality can certainly be seen as a form of Kantorovich duality, at least in a formal sense.
In this way, our positive results can be seen as rather surprising in that they establish a form of Kantorovich duality for highly irregular costs.
From the point of view of optimal transport, our results can be succinctly summarized as follows:
One can exchange topological regularity (the cost function is lower semi-continuous) for descriptive set-theoretic regularity (the cost function is supported on an equivalence relation) while maintaining a form of Kantorovich duality.

\subsection{Statement of Main Results}

Having established these precedents for duality between \eqref{eqn:primal} and \eqref{eqn:dual}, we now state the main results of the paper.
Throughout, let $(\Omega,\F)$ be an abstract measurable space.

Write $\mathcal{P}(\Omega,\F)$ for the set of probability measures on $(\Omega,\F)$.
For $\P,\P'\in \mathcal{P}(\Omega,\F)$, write $\Pi(\P,\P')$ for the space of all couplings of $\P,\P'$, that is, all probability measures on $(\Omega\times\Omega,\F\otimes\F)$ with marginals given by $\P$ and $\P'$, respectively.
Then, we consider pairs $(E,\G)$, where $E$ is an equivalence relation on $\Omega$ and where $\G$ is a sub-$\sigma$-algebra of $\F$.
We say that $E$ is \textit{measurable} if $E\in\F\otimes\F$ when viewed as a subset of $\Omega\times\Omega$, and we say that a pair $(E,\G)$ satisfies \textit{strong duality} if $E$ is measurable and if we have
\begin{equation}\label{eqn:strong-dual}
	\max_{A\in \G}|\P(A) - \P'(A)| = \min_{\tilde{\P}\in\Pi(\P,\P')}(1-\tilde{\P}(E))
\end{equation}
for all $\P,\P'\in\mathcal{P}(\Omega,\F)$; the appearence of ``max'' and ``min'' assert that the supremum and infimum are both achieved.

Our main results, as is the common practice in mathematical optimization, come in the form of sufficient conditions for strong duality to hold.
To state them, we need to introduce an important simplification. 
For an equivalence relation $E$ on $\Omega$, define the sub-$\sigma$-algebra 
\begin{equation*}
	E^{\ast} := \{A\in\F: \forall(\omega,\omega')\in E: (\omega\in A \Leftrightarrow \omega'\in A) \}
\end{equation*}
of $\F$, and, for sub-$\sigma$-algebra $\G$ of $\F$, define the equivalence relation
\begin{equation*}
	\G^{\ast} := \{(\omega,\omega')\in \Omega\times\Omega: \forall A\in \G:(\omega\in A \Leftrightarrow \omega'\in A) \}
\end{equation*}
on $\Omega$.
(An equivalent definition of $\G^{\ast}$ which has been extensively studied in classical works on the existence of regular conditional distributions \cite{BlackwellDubins,BertiRigo,HoffmanJorgensen} is $\G^{\ast} = \bigcup_{H \textrm{ is a $\G$-atom}}(H\times H)$, where by a \textit{$\G$-atom} we mean an intersection of all events in $\G$ containing a given point.)
It turns out (Lemma~\ref{lem:Galois-basics}) that $\ast$ is an antitone Galois correspondence which plays a central role in the notion of strong duality.

With the correspondence in hand, we can state our first main result:

\begin{customthm}{\ref{thm:basic-dual}}
	If $(\Omega,\F)$ is a standard Borel space and if a pair $(E,\G)$ satisfies $E\subseteq \G^{\ast},\G\subseteq E^{\ast}$, and $E\in \G\otimes \G$, then $(E,\G)$ satisfies strong duality.
\end{customthm}

While the conditions $E\subseteq \G^{\ast}$ and $\G\subseteq E^{\ast}$ are usually trivial to verify, the condition $E\in\G\otimes\G$ is more complicated.
In fact, while this condition appears to be rather general, there are important examples of pairs satisfying strong duality for which this result does not apply; notably, the tail equivalence relation and the tail $\sigma$-algebra $(E_0,\tail)$ given above.
This suggests that strong duality is a more general phenomenon, and this leads us to our second main result:

\begin{customthm}{\ref{thm:ctble-dual}}
	If $(\Omega,\F)$ is a standard Borel space and if equivalence relations $E_1\subseteq E_2\subseteq\cdots$ on $\Omega$ and sub-$\sigma$-algebras $\G_1\supseteq\G_2\supseteq \cdots$ of $\F$ are such that for each $n\in\N$ the pair $(E_n,\G_n)$ satisfies strong duality, then the pair $(\bigcup_{n\in\N}E_n,\bigcap_{n\in\N}\G_n)$ satisfies strong duality.
\end{customthm}

It is now useful to make a few remarks about these main theorems.
First, we note that Theorem~\ref{thm:basic-dual} and Theorem~\ref{thm:ctble-dual} are together powerful enough to establish strong dualizability for most ``reasonable'' pairs of equivalence relations and sub-$\sigma$-algebras appearing in probability theory.
To illustrate this, we investigate three major applications (in stochastic calculus, point process theory, and random sequence simulation) in Section~\ref{sec:applications}.
Additionally, we show (Proposition~\ref{prop:pseudo-strong-dual}) that many existing results in ergodic theory can be immediately upgraded into results about strong duality.
For the remaining remarks, let $(\Omega,\F)$ denote a standard Borel space, $(E,\G)$ a pair satisfying strong duality, and $\P,\P'\in\mathcal{P}(\Omega,\F)$ any probability measures.

Second, it is crucial to emphasize that the statements and proofs of our main results are purely measure-theoretic.
Even though expanding the generality to this setting forces one to leave behind many useful topological tools, our proofs reveal that only a few non-elementary results (notably, the existence of regular conditional probabilities, the Hahn-Jordan decomposition for finite signed measures, and some basic aspects of Hilbert spaces) are needed.
Keeping this generality in mind, it is highly non-trivial that optimizers can even be guaranteed to exist; since compactness arguments are not available, we must instead directly construct the putative optimizers and show that they have the desired properties.
Thus, it is exactly this measure-theoretic perspective that allows us to push beyond the reach of standard Kantorovich duality.

Third, we comment on the structure of solutions to \eqref{eqn:primal} and \eqref{eqn:dual}.
Indeed, for \eqref{eqn:primal}, the objective $\tilde{\P}\mapsto 1-\tilde{\P}(E)$ is affine and the feasible region $\Pi(\P,\P')$ is convex, so it follows that its solution set is convex and, by strong duality, non-empty;
in fact, we have an exact characterization of minimizers (Proposition~\ref{prop:opt-cond}), which largely parallels a classical characterization of primal minimizers for \eqref{eqn:simple-duality}.
For \eqref{eqn:dual}, it is known (Lemma~\ref{lem:dual-max}) that maximizers always exist, even without the assumption of strong duality; however, there does not appear to be a useful characterization of maximizers.
We also remark that, while our results give a very wide guarantee of the existence of solutions, concrete settings of interest will generally require further analysis in order to show that there exist solutions with desirable properties.

We now consider the perspective that one of \eqref{eqn:primal} and \eqref{eqn:dual} is a given \textit{primal} problem and that one seeks a suitable \textit{dual} problem of the opposite form.
As we outline in this last part, both perspectives are possible in great generality, although the situation is not completely symmetric.

On the one hand, we consider (Subsection~\ref{subsec:equiv}) the setting that an equivalence relation $E$ on $\Omega$ is given, and that one seeks a suitable sub-$\sigma$-algebra $\G$ of $\F$ such that $(E,\G)$ satisfies strong duality.
In this setting, it can be shown (Proposition~\ref{prop:strong-dual}) that $(E,E^{\ast})$ satisfies strong duality as soon as $(E,\G)$ is satisfies strong duality for some sub-$\sigma$-algebra $\G$ of $\F$.
Since $E^{\ast}$ is thus a canonical choice of $\G$, we can say that $E$ itself is \textit{strongly dualizable} whenever the pair $(E,E^{\ast})$ satisfies strong duality.
Let us also recall a standard term from descrptive set theory, that an equivalence relation $E$ on a standard Borel space $(\Omega,\F)$ is called \textit{smooth} if there exists a standard Borel space $(X,\mathcal{X})$ and a measurable map $\phi:(\Omega,\F)\to (X,\mathcal{X})$ such that for all $\omega,\omega'\in \Omega$ we have $(\omega,\omega')\in E$ if and only if $\phi(\omega) = \phi(\omega')$.
Then we have the following:

\begin{customcor}{\ref{cor:main-eq}}
	On a standard Borel space, any equivalence relation that can be written as a countable increasing union of smooth equivalence relations is strongly dualizable.
\end{customcor}

An equivalence relation that can be written as a countable increasing union of smooth equivalence relations is typically called \textit{hypersmooth} in descriptive set theory.
Indeed, many equivalence relations encountered in probability theory are hypersmooth.

On the other hand, we consider (Subsection~\ref{subsec:salg}) the setting that a sub-$\sigma$-algebra $\G$ of $\F$ is given, and that one seeks a suitable equivalence relation $E$ on $\Omega$ such that $(E,\G)$ satisfies strong duality.
As before we still say that $\G$ itself is \textit{strongly dualizable} whenever the pair $(\G^{\ast},\G)$ satisfies strong duality.
This setting is slightly more complicated than the previous, since $\G$ being strongly dualizable need not be equivalent to the existence of some $E$ such that $(E,\G)$ satisfies strong duality, and since there exist (Example~\ref{ex:asymmetry}) sub-$\sigma$-algebras that are not stongly dualizable.
Nonetheless, we have the following:

\begin{customcor}{\ref{cor:ctble-dual-salg}}
	On a standard Borel space, any sub-$\sigma$-algebra that can be written as a countable decreasing intersection of countably-generated sub-$\sigma$-algebras is strongly dualizable.
\end{customcor}

A sub-$\sigma$-algebra that can be written as a countable decreasing intersection of countably-generated sub-$\sigma$-algebras is sometimes called a \textit{tail} $\sigma$-algebra in probability theory \cite{DubinsHeath}.
(Note that we say ``a tail $\sigma$-algebra'' rather than ``the tail $\sigma$-algebra'', in order to distinguish the general case from the particular example of $\tail$.)
Indeed, many sub-$\sigma$-algebras encountered in probability theory are tail $\sigma$-algebras.

Finally, we note that, in addition to recovering some known results, our main theorems have novel consequences in a few probabilistic settings.
Most importantly, we obtain a characterization of the solution to the so-called ``Brownian germ coupling problem'' for one-dimensional diffusions (Theorem~\ref{thm:SDE-GCP}) which complements recent work in \cite{MEXIT,Vollering}.
We also obtain (Corollary~\ref{cor:DPP-sim-pot}) some partial results on couplings of point processes; these results are simple (and perhaps already known) in the setting of Poisson point processes, but our robust method of proof allows us to conclude the same about the more complicated setting of determinantal point processes.
Our work also yields (Corollary~\ref{cor:ress}) a non-constructive proof of the existence of some interesting randomized sorting algorithms, but many questions remain.

\subsection{Future Work}

We now outline a few open questions related to this work.

A main open question remaining at the end of this work concerns the generality of strong duality.
That is, can we find necessary and sufficient conditions for a pair $(E,\G)$ to satisfy strong duality?
While the sufficient conditions of this paper are already powerful enough to prove strong duality for most applications of interest to probabilists, we have two motivations for posing the question more generally.

The first motivation concerns the question of strong dualizability for equivalence relations.
In this case, we have shown that all hypersmooth equivalence relations are strongly dualizable, and also (Corollary~\ref{cor:CBER-dual}) that all countable Borel equivalence relations are strongly dualizable.
Additionally, we were unable, despite trying for quite some time, to come up with a single example of a measurable equivalence relation which is not strongly dualizable.
Thus, from the perspective of descriptive set theory, the following is natural:

\begin{conjecture}
	On a standard Borel space, all Borel equivalence relations are strongly dualizable.
\end{conjecture}

The second motivation concerns the question of strong dualizability for sub-$\sigma$-algebras.
Indeed, in the general theory of couplings, one has Thorisson's ``Working Hypothesis'' from  \cite{Thorisson} that ``Each meaningful distributional relation should have a coupling counterpart''.
The present results extend the existing work of \cite{PratelliRigo} in developing the theme that, if two probability measures agree on a given sub-$\sigma$-algebra, then they can be coupled to be almost surely equivalent under a suitable equivalence relation.
While there indeed exist sub-$\sigma$-algebras which are not strongly dualizable, we believe the present work shows that the term ``meaningful'' in fact excludes very few cases.

Another possible avenue for future work concerns a notion of duality which lies in between weak duality and strong duality; that is, that a pair $(E,\G)$ has the property that we have
\begin{equation*}
	\max_{A\in \G}|\P(A)-\P'(A)| = \inf_{\tilde{\P}\in \Pi(\P,\P')}(1-\tilde{\P}(E))
\end{equation*}
for all probability measures $\P,\P'\in \mathcal{P}(\Omega,\F)$.
(Recall that the supremum on the left is always achieved, by Lemma~\ref{lem:dual-max}.)
While this property is likely interesting in its own right, we did not pursue it in the present work, since most interesting applications to probability require the infimum to be realized.

Additionally, several open questions arise when one interprets the present results in the setting of optimal transport theory.
That is, we can consider slight modifications to the primal problem \eqref{eqn:primal} and ask whether there exists a suitable modification to the dual problem \eqref{eqn:dual} such that one maintains strong duality.
For one, we can fix a family of, say, bounded measurable functions $\{f_i\}_{i\in I}$ from $\Omega\times\Omega$ to $\R$ and consider the problem
\begin{equation*}
	\begin{split}
		\text{minimize} & \qquad 1-\tilde{\P}(E) \\
		\text{over} &\qquad \text{all couplings } \tilde{\P} \text{ of } \P,\P' \\
		\text{subject to } &\qquad \int_{\Omega\times\Omega}f_i\, d\tilde{\P} = 0 \text{ for all } i\in I.
	\end{split}
\end{equation*}
This constraint encodes, for example, the so-called problems of \textit{martingale optimal transport} \cite{MOT1,MOT2} and \textit{optimal transport under symmetry} \cite{Lopes}.
(See \cite{Zaev}.)
Another possibility is to consider, for a parameter $\varepsilon>0$, the regularized problem
\begin{equation*}
	\begin{split}
		\text{minimize} & \qquad 1-\tilde{\P}(E) + \varepsilon H(\tilde{\P}\, | \, \P\otimes \P') \\
		\text{over} &\qquad \text{all couplings } \tilde{\P} \text{ of } \P,\P' \\
	\end{split}
\end{equation*}
where $H(\tilde{\P}\, | \, \P\otimes \P')$ denotes the relative entropy.
This is the analog of the so-called problem of \textit{entropically-regularized optimal transport} which has been extensively studied in many recent works \cite{EntropicOT1,EntropicOT2}.

Lastly, we pose the question of whether one can use Kantorovich duality alone to prove strong duality for a sufficiently rich class of examples.
As we have seen in Subsection~\ref{subsec:Kantorovich}, this is not straightforward from existing literature, but it still seems possible to sharpen various approximation arguments in order to conclude.
Interestingly, it appears that arguments based on Kantorovich duality will lead to slightly different results than the arguments herein.
For example, we believe that some clever applications of Kantorovich duality should allow one to prove:

\begin{conjecture}\label{conj:f-sigma}
	On a Polish space, all $F_{\sigma}$ equivalence relations are strongly dualizable.
\end{conjecture}

Strictly speaking, this result is not implied by Corollary~\ref{cor:main-eq}, since there exist equivalence relations which can be written as a countable increasing union of closed sets but not as an countable increasing union of closed equivalence relations
(for instance, the shift equivalence relation $E_S$, and also many orbit equivalence relations).
Yet, for most applications, both Corollary~\ref{cor:main-eq} and Conjecture~\ref{conj:f-sigma} apply.

\section{Applications}\label{sec:applications}

In this section we apply our main results to three different areas of probability theory, and we hope that our notion of strong duality and our sufficient conditions will inspire similar work in the future.
Throughout this section, if $\Omega$ is a Polish space, we write $\mathcal{P}(\Omega)$ for the space of Borel probability measures on $\Omega$.

\subsection{Stochastic Calculus}

Our first application concerns stochastic calculus, so for this subsection we assume that the reader has familiarity with the basic theory, as outlined in, say, \cite{RevuzYor}.
For the setting, let us write $\Omega:=C_0([0,\infty),\R)$ for the usual Wiener space of continuous real-valued functions vanishing at zero, endowed with the topology of uniform convergence on compact sets.
Also write $\W$ for the Wiener measure on $\Omega$, that is, the law of a standard Brownian motion.

A common question in the theory of stochastic processes is that of determining when two stochastic processes behave similarly at small time scales.
For example, it is known that any process which is locally absolutely continuous with respect to Brownian motion must have the local limit of a Brownian motion \cite[Lemma~4.3]{DauvergneSarkarVirag} and that ``the compound Poisson component [of the L\'evy-It\^o decomposition] does not contribute to the initial sample path behaviour of a L\'evy process'' \cite[p. 15]{Bertoin}.
Another well-studied notion of small-time similarity (which is especially important from the point of view of mathematical finance) is that of separating times \cite{SepTimeI,SepTimeII}: For each pair of stochastic processes started at the same point, there exists a unique stopping time such that their laws are mutually absolutely continuous before this time and mutually singular after it; in fact, exact forms of the separating times are known for pairs of processes such as Bessel processes, L\'evy processes, and general diffusions.

Our object of interest in this application is a somewhat different notion of small-time similarity.
For the sake of simplicity, we focus on the case of comparing an arbitrary continuous process to a Brownian motion, in the following sense:

\begin{definition}
	A probability measure $\P\in \mathcal{P}(\Omega)$ is said to have the \emph{Brownian germ coupling property (Brownian GCP)} if one can construct a probability space $(\tilde{\Omega},\tilde{\F},\tilde{\P})$ on which are defined
	\begin{itemize}
		\item[(i)] a standard Brownian motion $B= \{B_t\}_{t\ge 0}$,
		\item[(ii)] a stochastic process $X = \{X_t\}_{t\ge 0}$ with law $\P$, and
		\item[(iii)] a random time $T$ with $\tilde{\P}(T>0) = 1$,
	\end{itemize}
	such that $\tilde{\P}(X_t = B_t \text{ for all } 0 \le t \le T) = 1$.
	A stochastic process is said to have the Brownian GCP if its law has the Brownian GCP.
\end{definition}

In words, a stochastic process with the Brownian GCP is a process which can be coupled to almost surely travel alongside a Brownian motion for some positive amount of time.
This notion of local Brownianity appears rather strong, and, apart from trivialities, it is not at all obvious how to construct stochastic processes satisfying this property.
Thus, we were greatly intrigued by the results of \cite{MEXIT}, which show that all Brownian motions with drift satisfy the Brownian GCP.
This suggests that the class of stochastic processes satisfying the Brownian GCP may actually be quite rich; our present goal is to understand exactly how rich it is.

Our major contribution is to show that the duality theory presents a robust approach to this problem.
Towards this end, if we write $\omega = \{\omega_t\}_{t\ge 0}$ for the canonical coordinate process in $\Omega$ and if we define the usual \textit{germ $\sigma$-algebra} via
\begin{equation*}
	\F_{0+} := \bigcap_{t > 0}\sigma(\omega_s: 0\le s \le t),
\end{equation*}
then we are led to the following fundamental result:

\begin{proposition}\label{prop:BGCP-characterization}
	A probability measure $\P\in \mathcal{P}(\Omega)$ has the Brownian GCP if and only if we have $\P(A) = \W(A)$ for all $A\in \F_{0+}$.
\end{proposition}

\begin{proof}
	It can be easily checked that the equivalence relation
	\begin{equation*}
		E_{0+} := \bigcup_{t>0}\{(\omega,\omega')\in \Omega\times \Omega:  \{\omega_s\}_{0\le s\le t} = \{\omega_s'\}_{0\le s\le t}\}
	\end{equation*}
	is hypersmooth and satisfies $E_{0+}^{\ast} = \F_{0+}$, hence by Corollary~\ref{cor:main-eq} that $(E_{0+},\F_{0+})$ satisfies strong duality.
	Thus, we claim that a probability measure $\P\in \mathcal{P}(\Omega)$ has the Brownian GCP if and only if there exists some coupling $\tilde{\P}\in \Pi(\W,\P)$ with $\tilde{\P}(E_{0+})=1$.
	For one direction, observe that, if $\P$ satisfies the GCP, then the joint law $\tilde{\P}$ of $(B,X)$ is exactly a coupling $\tilde{\P}\in \Pi(\W,\P)$ with $\tilde{\P}(E_{0+})=1$.
	For the other direction, suppose that $\tilde{\P}\in \Pi(\W,\P)$ is some coupling with $\tilde{\P}(E_{0+})=1$.
	We claim that the desired space $(\tilde{\Omega},\tilde{\F},\tilde{\P})$ is given by $\tilde{\Omega}:= \Omega\times\Omega$, and $\tilde{\F} := \mathcal{B}(\Omega)\otimes\mathcal{B}(\Omega)$, and $\tilde{\P}$ as given.
	Indeed, one can define $T:=\sup\{t \ge 0: \omega_s = \omega_s' \textrm{ for all } 0 \le s\le t \}$, which is clearly a measurable functions of the pair $(\omega,\omega')$.
	Thus, the result holds by the definition of strong duality.
\end{proof}

On the one hand, this provides an authoritative answer to the question of which stochastic processes have the Brownian GCP.
On the other hand, it remains to show that this equivalent condition is easy to verify in some concrete cases.
In fact, the value of Proposition~\ref{prop:BGCP-characterization} is that the classical tools of stochastic calculus have very much to say about the germ $\sigma$-algebra $\F_{0+}$, and thus many soft arguments suddently become available to us.

As an illustration of this, we have a simple and comprehensive characterization of the Brownian GCP for the case of time-homogeneous one-dimensional diffusions within the realm of classical conditions ensuring the existence of strong solutions to a given SDE \cite[Chapter~IX, Theorem~2.1]{RevuzYor}.
Since the Brownian GCP is spatially translation-invariant and temporally local, we lose no generality in restricting our attention to diffusions started at the origin and viewed on a finite time interval.

\begin{theorem}\label{thm:SDE-GCP}
	If $\mu:\R\to\R$ and $\sigma:\R\to[0,\infty)$ are Lipschitz continuous, then the strong solution $X = \{X\}_{0 \le t\le 1}$ of the SDE
	\begin{equation}\label{eqn:SDE}
		\begin{cases}
			dX_t = \mu(X_t)dt + \sigma(X_t)dB_t \textrm{ for } 0 \le t \le 1\\
			X_0 = 0,
		\end{cases}
	\end{equation}
	has the Brownian GCP if and only if $\sigma\equiv 1$ on some neighborhood of $0$.
\end{theorem}

\begin{proof}
	For one direction, let $(\tilde{\Omega},\tilde{\F},\tilde{\P})$ be a probability space witnessing the fact that $X$ satisfies the Brownian GCP.
	Then consider the event
	\begin{equation*}
		A:= \bigcup_{n\in\N}\{\omega\in \Omega: \langle\omega,\omega\rangle_{t} = t \textrm{ for all } [0,2^{-n}] \}
	\end{equation*}
	where $\langle\omega,\omega\rangle$ denotes the It\^{o} quadratic variation of $\omega$.
	Since $A\in\F_{0+}$ and $\tilde{\P}(B\in A) = 1$, Proposition~\ref{prop:BGCP-characterization} implies that we must have $\tilde{\P}(X\in A) = 1$.
	Now note that the quadratic variation of $X$ is given by $\langle X,X\rangle_t = \int_{0}^{t}(\sigma(X_s))^2\, ds$, so we conclude that there almost surely exists some $N\in\N$ with $\sigma(X_s) = 1$ for all $0\le s\le 2^{-N}$.
	In particular, we conclude $\sigma(0) =1$.
	
	Next we define the random times $\tau^{-}:=\inf\{t > 0: X_t < 0\}$ and $\tau^{+}:=\inf\{t > 0: X_t > 0\}$ on $(\tilde{\Omega},\tilde{\F},\tilde{\P})$, which are stopping times with respect to the natural filtration of $X$.
	We recall that the standard small-time approximation for diffusions with Lipschitz coefficients guarantees that $X_t/\sqrt{t}$ converges in distribution as $t\to 0$ to a Gaussian random variable with mean 0 and variance $\sigma(0) = 1$, hence
	\begin{align*}
		\tilde{\P}(\tau^{-} = 0) &= 
		\lim_{t\to 0}\tilde{\P}(X_s < 0 \text{ for some }0 \le s \le t) \\
		&\ge \liminf_{t\to\infty}\tilde{\P}(X_t < 0) \\
		&= \liminf_{t\to\infty}\tilde{\P}\left(\frac{X_t}{\sqrt{t}} < 0\right) = \frac{1}{2} > 0.
	\end{align*}
	By Blumenthal's zero-one law applied to the strong Markov process $X$, we see that $\tilde{\P}(\tau^{-} = 0) > 0$ implies $\tilde{\P}(\tau^{-} = 0) = 1$.
	The same argument applies to show  $\tilde{\P}(\tau^{+} = 0) = 1$.
	
	Finally, we put all the pieces together.
	By the almost sure continuity of $X$ and the fact that $\tilde{\P}(\tau^{-} = 0) = \tilde{\P}(\tau^{+} = 0) = 1$, we conclude that $\{X_s: 0 \le s \le T\}$ contains an open neighborhood of $0$ for any random variable $T$ which is almost surely strictly positive.
	Since we already know that there exists an $\N$-valued random variable $N$ with $\sigma(X_s) = 1$ for all $0\le s\le 2^{-N}$, we conclude that we have $\sigma \equiv 1$ on some neighborhood of $0$.
	
	For the other direction, suppose that $\sigma \equiv 1$ on some neighborhood $U$ of $0$.
	Then get reals $a,b\in\R$ with $a<b$ and $0\in (a,b)\subseteq [a,b]\subseteq U$.
	Since \eqref{eqn:SDE} admits strong solutions, we can construct $X=\{X_t\}_{t\ge0}$ on the probability space $(\Omega,\mathcal{B}(\Omega),\W)$ with its natural filtration $\{\F_t\}_{t\ge 0}$ defined via $\F_t := \sigma(\omega_s: 0 \le s\le t)$ for $t \ge 0$.
	We write $\E$ for the expectation on this space.
	
	To begin, we claim that the exit time $\tau_{a,b}^X:=\inf\{t > 0: X_t\notin [a,b]\}$ has finite exponential moments of all orders.
	To see this, define $m:= \max_{-a\le x\le b}\mu(x)$, which is finite by the continuity of $\mu$ on $[a,b]$.
	Since we have $\mu(X_t)\le m$ for all $0\le t\le \tau_{a,b}^X$ almost surely, it follows that one can construct $B^{m} = \{B_{t}^{m}\}_{t\ge0}$ a Brownian motion with drift $m$ on the same probability space in such a way that we have $X_t \le B_{t}^{m}$ for all $0\le t\le \tau_a^X$ almost surely.
	Then define the stopping time $\tau_a^{B^m} :=\inf\{t > 0: B_t^{m}\le a\}$, and note that we have $\tau_{a,b}^X\le \tau_a^{B^m}$ almost surely.
	It is known that $\tau_a^{B^m}$ has finite exponential moments of all orders, so the same must be true of $\tau_{a,b}^X$.
	
	In particular, we have shown
	\begin{equation*}
		\E\left[\exp\left(\frac{1}{2}\int_{0}^{\tau_{a,b}^X}(\mu(X_s))^2\,ds\right)\right] \le \E\left[\exp\left(\frac{m^2}{2}\tau_{a,b}^{X}\right)\right] <\infty.
	\end{equation*}
	This stopping-time version of Novikov's condition implies a stopping-time version of Girsanov's theorem which states that the law of $X$ is mutually absolutely continuous with the law of $X' = \{X_t'\}_{t\ge 0}$ the strong solution to the SDE
	\begin{equation*}
		\begin{cases}
			dX_t' = \sigma(X_t')dB_t \textrm{ for } 0 \le t \le 1\\
			X_0' = 0,
		\end{cases}
	\end{equation*}
	when restricted to the stopped $\sigma$-algebra $\F_{\tau_{a,b}}\subseteq \mathcal{B}(\Omega)$, for $\tau_{a,b}:=\inf\{t > 0: \omega_t \notin[a,b]\}$.
	Importantly, note $\F_{0+}\subseteq \F_{\tau_{a,b}}$ by the continuity of sample paths.
	Since Blumenthal's zero-one law implies that all events in $\F_{0+}$ have probability in $\{0,1\}$ under the laws $\W(X\in \cdot)$ and $\W(X'\in \cdot)$, we conclude $\W(X\in A)=\W(X'\in A)$ for all $A\in\F_{0+}$.
	
	Finally, observe that $\sigma\equiv 1$ on $U$ implies that $X'$ is a standard Brownian motion up to the stopping time $\tau_U := \inf\{t > 0: X'\notin U \}$.
	Since $\tau_U > 0$ by the continuity of sample paths, this means $X'$ has the Brownian GCP.
	Now by Proposition~\ref{prop:BGCP-characterization} we get $\W(X'\in A)=\W(A)$ for all $A\in\F_{0+}$, and, consequently, $\W(X\in A)=\W(A)$ for all $A\in\F_{0+}$.
	Therefore, one last application of Proposition~\ref{prop:BGCP-characterization} implies that $X$ has the Brownian GCP.
\end{proof}

Observe that Theorem~\ref{thm:SDE-GCP} recovers the result of \cite{MEXIT}, since Brownian motion with drift is certainly included as a special case.
However, the generality comes at a cost, since our result is essentially non-constructive.
In contrast, the work in \cite{MEXIT} is satisfyingly concrete, being based on William's excursion theory for Brownian motions with drift.
Thus, we would be greatly interested, for certain diffusions, in finding concrete couplings alongside a Brownian motion such that they agree for some positive random initial amount of time.
For instance, is this possible for a standard Brownian motion and a standard Ornstein-Uhlenbeck process?

\subsection{Point Process Theory}

Our second application concerns point processes, for which \cite{KallenbergMeas} is the standard reference.
To set things up, fix $k\in\N$ with $k\ge 1$, and let $\Omega$ denote the space of non-negative, locally finite, integer-valued, simple Borel measures on $\R^k$.
It is known \cite[Proposition~1.11]{Preston} that $\Omega$ is a standard Borel space when endowed with the Borel $\sigma$-algebra of the weak topology.
(In fact, $\R^k$ can be replaced with any locally compact Polish space).
An element of $\Omega$ is often called a \textit{configuration}, and it can equivalently be regarded random locally finite countable subset of $\R^k$.
Each $\P\in \mathcal{P}(\Omega)$ is called (the law of) a \textit{point process}.

The Poisson point processes are of course the simplest family of point processes, although recent attention has mostly turned to point processes exhibiting strong spatial correlations.
Recent work has significantly improved understanding of such aspects as Palm conditioning \cite{GhoshPalm,BufetovOlshanski,OsadaShirai}, number rigidity \cite{GhoshPeres,GhoshKrishnapur}, connections to infinite-dimensional stochastic differential equations \cite{Osada,HondaOsada}, and more.
Much of this progress concerns the so-called determinantal point processes (including those arising from the Gamma, Airy, Bessel, and Bergman kernels) for which many analytic and algebraic properties form essential tools; however, there have also been significant advances in the more general setting, notably including the point process of zeros of the standard planar Gaussian analytic function.

The statements of our results require some notation and terminology.
First, a finite collection of distinct points $\mathbf{x}\in(\R^k)^{\ell}$ will be called \textit{point set} and $\ell\in\N$ will be called its \textit{length}, denoted $|\mathbf{x}|$; here we use the convention that $\emptyset$ is the unique point set with $|\emptyset| = 0$.
Next, for a point process $\P\in\mathcal{P}(\Omega)$ and a point set $\mathbf{x} = (x_1,\ldots,x_{\ell}) \in(\R^k)^{\ell}$, we write $\P_{\mathbf{x}}$ for the \textit{reduced Palm measure} of $\P$ under $\mathbf{x}$; roughly speaking, the point process $\P_{\mathbf{x}}$ is the result of first conditioning $\P$ to contain the points $x_1,\ldots, x_{\ell}\in\R^k$ and then deleting them from the resulting configuration.
Then we introduce the following:

\begin{definition}
	With respect to a point process $\P\in \mathcal{P}(\Omega)$, two points sets $\mathbf{x}$ and $\mathbf{x}'$ are said to have \emph{similar potential} if there exists a coupling $\omega_{\mathbf{x}}$ and $\omega_{\mathbf{x'}}$ of $\P_{\mathbf{x}}$ and $\P_{\mathbf{x}'}$, respectively, on a probability space $(\tilde{\Omega},\tilde{\F},\tilde{\P})$ such that we have $\tilde{\P}(\|\omega_{\mathbf{x}}-\omega_{\mathbf{x}'}\|_{\TV} < \infty) = 1$.
\end{definition}

Let us now provide some interpetation for this notion and this terminology.
Indeed, consider taking a sample $\omega$ from $\P$ and transforming it into a sample $\omega_{\mathbf{x}}$ from $\P_{\mathbf{x}}$.
One of course has to modify some subset of the points of $\omega$ to achieve this, but the amount of total modification necessary depends on $\mathbf{x}$;
if it is extremely rare that $\mathbf{x}$ appears in $\omega$ then many points will need to be moved, and if it is extremely common that $\mathbf{x}$ appears in $\omega$ then not so many points will need to be moved.
Now let us consider the effect of conditionings on two different point sets, $\mathbf{x}$ and $\mathbf{x}'$.
If $\mathbf{x}$ is much rarer than $\mathbf{x}'$, then many points need to be moved in order to get from $\omega$ to $\omega_{\mathbf{x}}$ while not so many points need to be moved in order to get from $\omega$ to $\omega_{\mathbf{x}'}$.
Since the condition $\|\omega_{\mathbf{x}}-\omega_{\mathbf{x}'}\|_{\TV} < \infty$ exactly means that the configurations $\omega_{\mathbf{x}}$ and $\omega_{\mathbf{x}'}$ differ only in finitely many points, this means that $\mathbf{x}$ and $\mathbf{x}'$ have a similar probability of appearing in $\omega$.

Alternatively, one can view conditioning on containing $\mathbf{x}$ as adding some amount of potential energy to the configuration $\omega$; if the amount of energy added is very high the system is likely to find a completely different equilibrium, and if the amount of energy added is very low the system is likely to stay close to its current equilibrium.
Thus, the condition $\|\omega_{\mathbf{x}}-\omega_{\mathbf{x}'}\|_{\TV} < \infty$ can be interpreted as saying that $\mathbf{x}$ and $\mathbf{x}'$ add a comparable amount of potential energy to the system $\omega$.
In fact, the idea to consider the energy added to a point process by (Palm) conditioning is not a new one; this is usually made quantitative via the analysis of so-called \textit{logarithmic derivatives} \cite{logDeriv}, but only for the case of (Palm) conditioning on single points.

The contribution of this subsection is to show that the similarity of potentials is closely related, via our duality theory, to many aspects of the existing literature.
In order to do this, we need to introduce some notation.
For a compact set $B\subseteq \R^k$, let us write $\F_B := \sigma(\omega(U\setminus B): U\subseteq \R^k \textrm{ open})$ for the $\sigma$-algebra generated by the configuration $\omega$ outside $B$.
Then we define the \textit{tail $\sigma$-algebra} 
\begin{equation*}
	\tail := \bigcap_{\substack{B\subseteq \R^k \\ B \textrm{ compact}}}\F_{B},
\end{equation*}
and we prove the following:

\begin{proposition}\label{prop:tail-sim-pot}
	Two point sets $\mathbf{x}$ and $\mathbf{x}'$ have similar potential with respect to a point process $\P\in \mathcal{P}(\Omega)$ if and only if $\P_{\mathbf{x}}(A) = \P_{\mathbf{x}'}(A)$ for all $A\in \tail$.
\end{proposition}

\begin{proof}
	It can be easily shown that the equivalence relation
	\begin{equation*}
		E := \bigcup_{\substack{B\subseteq \R^k \\ B \textrm{ compact}}}\{(\omega,\omega')\in \Omega\times\Omega: \omega(U\setminus B) = \omega(U\setminus B) \textrm{ for all } U\subseteq \R^k \textrm{ open}\}
	\end{equation*}
	is hypersmooth and satisfies $E^{\ast} = \tail$, hence by Corollary~\ref{cor:main-eq} that $(E,\tail)$ satisfies strong duality.
	Moreover, we observe that two point sets $\mathbf{x}$ and $\mathbf{x}'$ have similar potential with respect to a point process $\P$ if and only if there exists a coupling $\tilde{\P}\in\Pi(\P_{\mathbf{x}},\P_{\mathbf{x}'})$ with $\tilde{\P}(E) = 1$.
	Therefore, the result follows from the definition of strong duality.
\end{proof}

To push this one step further, we provide the following consequence for determinantal point processes.

\begin{corollary}\label{cor:DPP-sim-pot}
	If $\P$ is a determinantal point process and the point sets $\mathbf{x}$ and $\mathbf{x}'$ are such that the reduced Palm measures $\P_{\mathbf{x}}$ and $\P_{\mathbf{x}'}$ are mutually absolutely continuous, then $\mathbf{x}$ and $\mathbf{x}'$ have similar potential with respect to $\P$.
\end{corollary}

\begin{proof}
	By the Shirai-Takahashi theorem \cite[Theorem~1.7]{ShiraiTakahashi}, reduced Palm measures of determinantal point processes are themselves determinantal point processes.
	Moreover, we have the the Osada-Osada-Lyons zero-one law\footnote{There appears to be some amgibuity about whose names should appear with this result:
		The zero-one law was first proven for discrete spaces and conjectured for more general spaces by Lyons in \cite{LyonsDeterminantal}. Subsequently, it was proven more generally (under very mild assumptions) by Osada and Osada in \cite{OsadaOsada}.
		However, the proof in \cite{OsadaOsada} contained a small error which was later repaired by Lyons in \cite{LyonsTail}. Also see \cite[Theorem~1.7]{BufetovQiuShamov} by Bufetov, Qiu, and Shamov for closely-related concurrent work.} which states that $\tail$ is trivial under every determinantal point processes.
	Consequently, $\P_{\mathbf{x}}$ and $\P_{\mathbf{x}'}$ being mutually absolutely continuous implies that we must have $\P_{\mathbf{x}}(A) = \P_{\mathbf{x}'}(A)$ for all $A\in\tail$.
	Thus, the proof is completed by applying Proposition~\ref{prop:tail-sim-pot}.
\end{proof}

To close this subsection, we explore some consequences of combining Corollary~\ref{cor:DPP-sim-pot} with existing literature.
First, we note that many point processes of interest have the property that the reduced Palm measures $\P_{\mathbf{x}}$ and $\P_{\mathbf{x}'}$ are mutually absolutely if and only if the point sets $\mathbf{x}$ and $\mathbf{x}'$ satisfy $|\mathbf{x}|=|\mathbf{x}'|$:
indeed, this is the case for many determintal point processes on $\R$ with integrable kernels \cite[Theorem~1.4]{Bufetov}, for the Ginibre point process \cite[Theorem~1.1]{OsadaShirai}, and for certain determinantal point processes on $\C$ associated to Hilbert spaces of entire functions \cite[Theorem~1.1]{BufetovQiu}.
Consequently, all of these point processes have the property that any two point sets have similar potential whenever they have the same length.
For a somewhat different example, we recall that certain determinantal point processes on the unit disc $\mathbb{D}\subseteq\C$ associated to Hilbert spaces of holomorphic functions have the property that the reduced Palm measures of any two point sets are mutually absolutely continuous \cite[Theorem~1.4]{BufetovQiu}.
Thus, these determinantal point processes have the property that all points sets have similar potential to each other, under no assumption on their lengths.

\subsection{Random Sequence Simulation}

Our final application concerns simulating random sequences with complicated dependency structures.
For this section, let $X$ be a Polish space, and set $\Omega:=X^{\N}$ with the Borel $\sigma$-algebra of the product topology on $\Omega$.
We write $\omega = \{\omega_n\}_{n\in\N}$ for an arbitrary element of $\Omega$.

%As motivation, let us make a simple observation.
%Suppose that $\mu$ is a fully-supported Borel probability measure on a finite set $X$ and that $\P:=\bigotimes_{n\in\N}\mu$ is the law of an independent identically distributed sequence of $X$-valued data each with law $\mu$.
%We claim that if $C\subseteq\Omega$ is a cylinder set $C = \{\omega_{n_0} = a_0,\ldots, \omega_{n_k} = a_k\}$ for some $k\in\N$, some strictly increasing sequence $n_0,\ldots, n_k\in\N$, and some values $a_0,\ldots, a_k\in X$, then there exists an algorithm which, when applied to a random sequence $\omega$ from $\P$, outputs a random sequence $\omega'$ from $\P(\cdot\,|\,C)$.
%Indeed, it is easy to construct such an algorithm directly:
%First, swap the character at position $n_0$ with the character at position $\min\{m_0 > n_k: \omega_{m_0}=a_0\}$.
%Then, recursively for $i\in\{1,\ldots, k\}$, swap the charater at position $n_i$ with the charater at position $\min\{m_i > m_{i-1}: \omega_{m_i}=a_i\}$.
%Note that this algorithm terminates in an almost surely finite number of steps (reads and swaps) and does not require any external randomness.
%\textcolor{red}{Actually, there is a small error here, since $\P(\cdot\, | \, C)$ is certainly not absolutely continuous with respec to $\P$.}

Our motivation is to try to understand which laws of random sequences can be simulated by applying a random sorting algorithm to an independent identically distributed sequence.
Towards making this precise, let us define some objects of interest.
First, write $S$ for the group of all permutations of $\N$ (that is, all bijections $\sigma:\N\to\N$), and let us endow $S$ with the topology of pointwise convergence.
Then for $n\in\N$, write $\Pi_n$ for the subgroup of all permutations $\sigma\in S$ with $\sigma(i) = i$ for $i > n$; these are the permutations that fix all elements except possibly those in $\{1,\ldots, n\}$.
We also define $\Pi_{\infty}:=\bigcup_{n\in\N}\Pi_n\subseteq S$; it is known that the topology inherited from $S$ makes $\Pi_{\infty}$ into a Polish group \cite[Theorem~6.26]{Turbulence}.
We also let $\Pi_{\infty}$ act on $\Omega$ in the natural way, by setting $\sigma\cdot\{\omega_n\}_{n\in\N} := \{\omega_{\sigma^{-1}(n)}\}_{n\in\N}$.
In particular, observe that the action is jointly continuous when viewed as a map $a:\Pi_{\infty}\times\Omega\to\Omega$.
When we wish to emphasize that one argument is fixed, we will equivalently write $a(\sigma,\omega) =: a_{\omega}(\sigma) =: a_{\sigma}(\omega)$ for $\omega\in\Omega$ and $\sigma\in \Pi_{\infty}$.

Now we introduce our main object of interest.

\begin{definition}
	For $\P,\P'\in \mathcal{P}(\Omega)$, an \emph{algorithmic reassortment from $\P$ to $\P'$} is a transition kernel $K:\Omega\times\mathcal{B}(\Pi_{\infty})\to [0,1]$ such that we have
	\begin{equation*}
		\P'(A) = \int_{\Omega}K(\omega,a_{\omega}^{-1}(A))\,d\P(\omega)
	\end{equation*}
	for all $A\in \mathcal{B}(\Omega)$;
	we say that $\P'$ \emph{is an algorithmic reassortment of} $\P$ if there exists an algorithmic reassortment from $\P$ to $\P'$.
\end{definition}

An algorithmic reassortment from $\P$ to $\P'$ exactly corresponds to an algorithm described at the beginning of this subsection:
For each $\omega\in\Omega$, we have $K(\omega,\cdot)$ which represents the required (possibly-random) permutation bringing $\omega$ to $\omega'$.
Since a random element from $K(\omega,\cdot)$ fixes all but finitely elements of $\N$ almost surely, it can be regarded as (in fact, decomposed into) a finite sequence of transpositions on $\N$, and this means that an algorithm represented by an algorithmic reassortment $K$ must terminate in a finite number of swaps.
However, such algorithms are not required to terminate in a finite number of reads; they may need to read the entire infinite sequence $\omega$ before deciding which swaps to apply.

Of course, it will turn out to be the case that algorithmic reassortment can be studied via our duality theory.
To do this, let us define the usual \textit{exchangeable $\sigma$-algebra} via
\begin{equation*}
	\exch:= \{A\in\mathcal{B}(\Omega): a_{\sigma}^{-1}(A) = A \textrm{ for all } \sigma\in \Pi_{\infty} \}.
\end{equation*}
Then we have the following:

\begin{theorem}\label{thm:reass-characterization}
	For probability measures $\P,\P'\in\mathcal{B}(\Omega)$, we have that $\P'$ is an algorithmic reassortment of $\P$ if and only if $\P(A) = \P'(A)$ for all $A\in \exch$.
\end{theorem}

\begin{proof}
	One can easily check that the equivalence relation
	\begin{equation*}
		E := \{(\omega,\omega')\in\Omega\times\Omega: \sigma\cdot\omega = \omega' \textrm{ for some } \sigma\in \Pi_{\infty} \},
	\end{equation*}
	is hypersmooth and satisfies $E^{\ast} = \exch$, so Corollary~\ref{cor:main-eq} implies that $(E,\exch)$ satisfies strong duality.
	(Alternatively, one can directly apply Corollary~\ref{cor:CBER-dual}.)
	
	By the definition of strong duality, it therefore suffices to show that $\P'$ being an algorithmic reassortment of $\P$ is equivalent to the existence of some $\tilde{\P}\in\Pi(\P,\P')$ with $\tilde{\P}(E) = 1$.
	For the first direction, suppose that $\P'$ is an algorithmic reassortment of $\P$, and let $K$ be the guaranteed kernel.
	Now construct a probability space on which are defined $\omega$ a random element of $\Omega$ with law $\P$ and $\sigma$ a random element of $\Pi_{\infty}$ whose law conditional on $\omega$ is $K(\omega,\cdot)$.
	Then, setting, $\omega':= \sigma\cdot\omega$, it is easy to see that the joint law $\tilde{\P}$ of $(\omega,\omega')$ satisfies $\tilde{\P}(E) = 1$.
	For the second direction, we suppose that there exists $\tilde{\P}\in\Pi(\P,\P')$ with $\tilde{\P}(E) = 1$.
	Now we claim that there exists a measurable map $\psi:E\to \Pi_{\infty}$ with $\psi(\omega,\omega')\cdot\omega = \omega'$ for all $(\omega,\omega')\in E$, where $E$ is given the Borel $\sigma$-algebra of the topology inherited from $\Omega\times\Omega$; indeed this can be verified easily from the Kuratowski and Ryll-Nardzewski measurable selection theorem, but we omit the details.
	(Note, however, that $E$ is not Polish and that $(E,\mathcal{B}(E))$ is not standard Borel.)
	It follows that the probability space $(\Omega\times\Omega,\mathcal{B}(\Omega)\otimes\mathcal{B}(\Omega),\tilde{\P})$ supports a random variable $\sigma:\Omega\times\Omega\to\Pi_{\infty}$ defined for $(\omega,\omega')\in \Omega\times\Omega$ via
	\begin{equation*}
		\sigma(\omega,\omega'):=\begin{cases}
			\psi(\omega,\omega'), &\textrm{ if } (\omega,\omega')\in E, \\
			\sigma_0, &\textrm{ if } (\omega,\omega')\notin E,
		\end{cases}
	\end{equation*}
	where $\sigma_{0}\in\Pi_{\infty}$ is any arbitrary fixed element.
	Since $\Omega$ and $\Pi$ are both Polish, there exists a regular conditional distribution $K:\Omega\times\mathcal{B}(\Pi_{\infty})\to [0,1]$ of $\sigma$ with respect to $\omega$, and this completes the proof.
\end{proof}

\begin{remark}
	The preceding result shows that $\P$ is an algorithmic reassortment of $\P'$ if and only if $\P'$ is an algorithmic reassortment of $\P$, which is not obvious from the definition.
\end{remark}

Our main application of this characterization is as follows:

\begin{corollary}\label{cor:ress}
	If $\P = \bigotimes_{n\in\N}\mu$ for some $\mu\in \mathcal{P}(X)$ and $V:\Omega\to \R$ is measurable, non-negative, and has $0 < \int_{\Omega}V\, d\P < \infty$, then $\P'\in \mathcal{P}(\Omega)$ defined via
	\begin{equation*}
		\frac{d\P'}{d\P}(\omega) := \frac{V(\omega)}{\int_{\Omega}V\, d\P}
	\end{equation*}
	is an algorithmic reassortment of $\P$.
\end{corollary}

\begin{proof}
	By the Hewitt-Savage zero-one law, we have $\P(A) \in\{0,1\}$ for all $A\in \exch$.
	Moreover, $\P'$ and $\P$ are mutually absolutely continuous by construction.
	Therefore, $\P'(A) = \P(A)$ for all $A\in\exch$, so the result follows from Theorem~\ref{thm:reass-characterization}.
\end{proof}

There are still a number of interesting questions in this direction.
For instance, while we have given a rather abstract existence proof of such algorithms, we would be interested in understanding whether they can be constructed, in some generality, in a more concrete way.
Moreover, we would be interested in understanding when such algorithms can be guaranteed to require only finitely many reads, in addition to the existing guarantee of only finitely many writes.

\section{Proofs of Main Results}

In this section we develop the abstract duality theory as outlined in the introduction.
%More specifically, in Subsection~\ref{subsec:galois} we give the basic definitions and constructions of this duality in the form of a Galois correspondence, in Subsection~\ref{subsec:strong-duality} we state and prove the main theorem, and in Subsection~\ref{subsec:duality-open} we give some open problems.
Throughout this section, $(\Omega,\F)$ denotes a fixed measurable space.

\subsection{Preliminaries}

In this subsection we review some notation, definitions, and results that will be needed throughout the paper.

The first concepts concern the elementary notions of relations.
By a \textit{relation} $R$ on $\Omega$ we mean any subset $R\subseteq\Omega\times\Omega$.
By an \textit{equivalence relation} $E$ on $\Omega$ we mean a relation $\Omega$ which is reflexive, symmetric, and transitive.
There is a well-known correspondence between equivalence relations on $\Omega$ and partitions of $\Omega$, where an equivalence relation gives rise to a partition by dividing the space into equivalence classes, and where a partition gives rise to an equivalence relation by declaring two points to be equivalent if and only if they occur in the same element of the partition.

Next we consider basic measure theory, as found in, say \cite{Kallenberg} or \cite{Bogachev}.
For this part, we let $(S,\S)$ denote an abstract measurable space; the results herein will typically be applied when $(S,\S)$ is taken to be $(\Omega,\F)$ or $(\Omega\times\Omega,\F\otimes \F)$, but some other cases will also be used.
For example, write $\R$ for the set of real numbers and $\mathcal{B}(\R)$ for the Borel $\sigma$-algebra of its standard topology.
We write $b\S$ for the space of bounded measurable functions from $(S,\S)$ to $(\R,\mathcal{B}(\R))$.
The Cartesian product of $S$ with itself is denoted $S\times S$, and the product $\sigma$-algebra of $\S$ with itself, that is the $\sigma$-algebra generated by the $\S\times\S$, is denoted $\S\otimes\S$.
We write $\pi,\pi':S\times S\to S$ for the projection maps onto the first and second coordinates, respectively.

Now, we write $\mathcal{M}(S,\S)$ for the space of finite signed mesaures on $(S,\S)$.
We endow $\mathcal{M}(S,\S)$ with the partial order $\le$ where $\mu,\mu'\in \mathcal{M}(S,\S)$ have $\mu \le \mu'$ if and only if we have $\mu(A)\le \mu'(A)$ for all $A\in\S$.
For $\mu,\mu'\in \mathcal{M}(S,\S)$ there exists \cite[Corollary~2.9]{Kallenberg} a unique element of $ \mathcal{M}(S,\S)$ which is $\le$-maximal among all elements which are $\le$-bounded above by both $\mu$ and $\mu'$, and we denote this by $\mu\wedge \mu'$.
In fact, if $P,N\in\S$ are respectively the positive part and negative part of the Jordan decomposition of the signed measure $\mu-\mu'$, then we have $\mu\wedge \mu' = \mu(\cdot \cap N) + \mu'(\cdot \cap P)$.

We write $\mathcal{P}(S,\S)$ for the space of probability measures on $(S,\S)$.
By a \textit{sub-probability measure} on $(S,\S)$ we mean an element $\mu\in \mathcal{M}(S,\S)$ with $0 \le \mu(A) \le 1$ for all $A\in\S$, and we write $\mathcal{P}_s(S,\S)\supseteq\mathcal{P}(S,\S)$ for the space of all sub-probability measures on $(S,\S)$.
For any $\mu,\mu'\in \mathcal{P}_s(S,\S)$, a \textit{coupling} of $\mu$ and $\mu'$ is an element $\nu\in\mathcal{P}_s(S\times S,\S\otimes\S)$ satisfying $\nu\circ \pi^{-1}= \mu$ and $\nu\circ (\pi')^{-1}= \mu'$, and a \textit{sub-coupling} of $\mu$ and $\mu'$ is an element $\nu\in\mathcal{P}_s(S\times S,\S\otimes\S)$ satisfying $\nu\circ \pi^{-1}\le \mu$ and $\nu\circ (\pi')^{-1}\le \mu'$.
The spaces of couplings and sub-couplings of $\mu,\mu'\in \mathcal{P}_s(S,\S)$ are denoted $\Pi(\mu,\mu')$ and $\Pi_s(\mu,\mu')$, respectively; observe that $\Pi(\mu,\mu')$ is empty unless $\mu(S) = \mu'(S)$.

Finally, we review some concepts related to Polish spaces; these ideas can be found in standard sources on measure theory like \cite{Bogachev}, or in more specialized treatments like \cite{KechrisClassical}.
By a \textit{Polish space} we mean a separable topological space $(X,\tau)$ such that there exists a complete metric $d$ on $X$ which generates the topology $\tau$.
We write $\mathcal{B}(\tau)$ for the \textit{Borel $\sigma$-algebra} of a Polish space, that is, the $\sigma$-algebra generated by the open sets.
By a \textit{standard Borel space} we mean a measurable space $(X,\mathcal{X})$ such that there exists a topology $\tau$ on $X$ making $(X,\tau)$ into a Polish space and such that we have $\mathcal{B}(\tau) =\mathcal{X}$

Let us also review our central definition; let $(\Omega,\F)$ denote an abstract measurable space, and consider any pair $(E,\G)$ where $E$ is a relation on $\Omega$ and $\G$ is a subset of $\F$.
We say that $E$ is \textit{measurable} if $E\in\F\otimes\F$.
We say that $(E,\G)$ satisfies \textit{strong duality} if $E$ is measurable and if we have
\begin{equation*}
	\max_{A\in \G}|\P(A) - \P'(A)| = \min_{\tilde{\P}\in\Pi(\P,\P')}(1-\tilde{\P}(E))
\end{equation*}
for all $\P,\P'\in \mathcal{P}(\Omega,\F)$.
An event $A\in\G$ achieving the supremum on the left is called \textit{$\G$-optimal for $\P,\P'$} and a coupling $\tilde{\P}\in \Pi(\P,\P')$ achieving the minimum on the right side is called \textit{$E$-optimal for $\P,\P'$}; a coupling $\tilde{\P}\in \Pi(\P,\P')$ satisfying $1-\tilde{\P}(A) = \sup_{A\in\G}|\P(A)-\P'(A)|$ will be called \textit{$(E,\G)$-optimal}, and note that $E$-optimality is equivalent to $(E,\G)$-optimality if $(E,\G)$ satisfies strong duality.
Further, a coupling satisfying $\tilde{\P}(E) = 1$ is called \textit{$E$-successful}.
It turns out that it will also be useful along the way to consider two further notions of duality; we briefly introduce them now, and in the following subsections we study them more carefully.

First, let us say that a pair $(E,\G)$ satisfies \textit{weak duality} if $E$ is measurable and if we have
\begin{equation}\label{eqn:weak-dual}
	\sup_{A\in \G}|\P(A) - \P'(A)| \le \inf_{\tilde{\P}\in\Pi(\P,\P')}(1-\tilde{\P}(E))
\end{equation}
for all $\P,\P'\in \mathcal{P}(\Omega,\F)$.
Let us emphasize that the infimum need not be achieved in that $E$-optimal couplings are not required to exist (and that we will later show that the supremum is always achieved.)
Moreover, observe that strong duality obviously implies weak duality.

Second, let us say that a pair $(E,\G)$ satisfies \textit{quasi-strong duality} if, for all $\P,\P'\in \mathcal{P}(\Omega,\F)$, the following are equivalent:
\begin{itemize}
	\item[(i)] For all $A\in \G$, we have $\P(A) = \P'(A)$.
	\item[(ii)] There exists a $\tilde{\P}\in \Pi(\P,\P')$ and a $N\in\F\otimes\F$ with $\tilde{\P}(N) = 0$ and $(\Omega\times\Omega)\setminus E\subseteq N$.
\end{itemize}
Crucially, observe that the measurability of $E$ is not required in order for $(E,\G)$ to satisfy quasi-strong duality;
if $E$ happens to be measurable, then (ii) above is simply equivalent to the existence of an $E$-successful coupling.
%Observe that pseudo-strong duality is exactly the ``qualitative'' form of strong duality given in prior results like \eqref{eqn:tail-duality} and \eqref{eqn:shift-duality}.
The primary desire to generalize beyond measurable relations is to be able to say something about analytic equivalence relations on standard Borel spaces, which are a common object of study in ergodic theory \cite{Georgii}.

\subsection{Weak Duality}\label{subsec:weak}

In this subsection we study weak duality as a stepping stone to strong duality.
The results herein provide some reductions which simplify our later work.

To begin, we state a fundamental set-theoretic correspondence between equivalence relations and sub-$\sigma$-algebras.

\begin{definition}\label{def:Galois}
	For any relation $E$ on $\Omega$, define the sub-$\sigma$-algebra 
	\begin{equation*}
		E^{\ast} := \{A\in\F: \forall(\omega,\omega')\in E: (\omega\in A \Leftrightarrow \omega'\in A) \}
	\end{equation*}
	of $\F$, and, for any subset $\G$ of $\F$, define the equivalence relation
	\begin{equation*}
		\G^{\ast} := \{(\omega,\omega')\in \Omega\times\Omega: \forall A\in \G:(\omega\in A \Leftrightarrow \omega'\in A) \}.
	\end{equation*}
	on $\Omega$.
\end{definition}

We point out that the collection $E^{\ast}$ is a $\sigma$-algebra even when $E$ is not an equivalence relation, and likewise the relation $\G^{\ast}$ is an equivalence relation even when $\G$ is not a $\sigma$-algebra.
Although these operations will typically be applied when $E$ and $\G$ are, respecively, an equivalence relation and a $\sigma$-algebra, the added generality will be useful in some cases.
We also have the following:

\begin{remark}\label{ref:invariant-salg-ambient}
	The map $E\mapsto E^{\ast}$ depends on the ambient $\sigma$-algebra $\F$ with which $\Omega$ is endowed, while the map $\G\mapsto \G^{\ast}$ does not.
	We hope it causes no confusion that this dependence is not emphasized in the notation.
\end{remark}

Another important collection of remarks concerns the algebraic structure of this correspondence.
While such algebraic terminology does not immediately pay dividends for our work, we believe that it is nonetheless useful to highlight the perspective;
we direct the reader to \cite[Section~1.6]{Blyth} for further information on abstract order theory.

\begin{lemma}\label{lem:Galois-basics}
	For $E$ a relation on $\Omega$ and $\G$ a subset of $\F$, we have
	\begin{itemize}
		\item[(i)] $E^{\ast\ast} \supseteq E$ and $\G^{\ast\ast} \supseteq \G$, and
		\item[(ii)] $E\subseteq\G^{\ast}$ if and only if $\G\subseteq E^{\ast}$.
	\end{itemize}
	In other words, the correspondence given by $\,^{\ast}$ is an antitone Galois correspondence.
\end{lemma}

As is the case in many antitone Galois correspondences, it is instructive to understand the operation arising by applying the correspondence twice.
(In general, the resulting operation is not the identity.)
%In what follows, we write $E^{\ast\ast} := (E^{\ast})^{\ast}$ and $\G^{\ast\ast} := (\G^{\ast})^{\ast}$ for relations $E$ on $\Omega$ and subsets $\G$ of $\F$.

\begin{lemma}\label{lem:equiv-double-dual}
	For $E$ a relation on $\Omega$ we have $E^{\ast\ast} = E$ if and only if $E$ is an equivalence relation.
\end{lemma}

\begin{proof}
	First, suppose that $E$ is an equivalence relation, and let us show that $E^{\ast\ast} = E$.
	Since we have $E\subseteq E^{\ast\ast}$ by Lemma~\ref{lem:Galois-basics}, it suffices to show $E^{\ast\ast}\subseteq E$.
	Indeed, take any $(\omega,\omega')\in E^{\ast\ast}$ and recall that for any $A\in E^{\ast}$ that $\omega\in A$ if and only if $\omega'\in A$.
	Now combine $E\in \F\otimes \F$ with Fubini's theorem to get $[\omega]_E := \{\omega''\in \Omega: (\omega,\omega'')\in \Omega\}\in\F$, then use the symmetry and transitivity of $E$ to get $[\omega]_E\in E^{\ast}$.
	By the reflexivity of $E$ we have $\omega\in [\omega]_E$, hence $(\omega,\omega')\in E^{\ast\ast}$ gives $\omega'\in [\omega]_E$.
	This is exactly $(\omega,\omega')\in E$, hence $E^{\ast\ast}\subseteq E$, as needed.
	Second, note that $E^{\ast\ast} = E$ implies that $E$ is an equivalence relation, since $E^{\ast\ast}$ is always an equivalence relation.
\end{proof}

As we see in the following, it is not true that $\G$ being a $\sigma$-algebra implies $\G^{\ast\ast}= \G$.

\begin{example}\label{ex:ctble-coctble}
	Let $\Omega$ be an uncountable set, let $\F$ be the $\sigma$-algebra of all subsets of $\Omega$, and let $\G$ be the $\sigma$-algebra of all sets which are countable or whose complements are countable.
	Then $\G^{\ast} = \Delta$ and $\G^{\ast\ast} = \F\supsetneq\G$.
\end{example}

Having clarified some aspects of the Galois correspondence, our next goal is to show that it is indeed useful in studying weak and strong duality.
As a first indication of this, let us show how the Galois correspondence leads to a necessary and sufficient condition for weak duality.

\begin{proposition}\label{prop:weak-dual}
	For $E$ a relation on $\Omega$ and $\G$ a subset of $\F$ with $\Omega\in\G$, the pair $(E,\G)$ satisfies weak duality if and only if $E$ is measurable and $E\subseteq \G^{\ast}$ and (equivalently, or) $\G\subseteq E^{\ast}$.
\end{proposition}

\begin{proof}
	First suppose that $(E,\G)$ satisfies weak duality.
	By definition, $E$ is measurable, so we only need to show $\G\subseteq E^{\ast}$ and $E\subseteq \G^{\ast}$.
	These are equivalent by part (ii) of Lemma~\ref{lem:Galois-basics}, so it suffices to show $\G\subseteq E^{\ast}$.
	Indeed, take any $(\omega,\omega')\in E$, and note that by setting $\P:= \delta_{\omega}$ and $\P':=\delta_{\omega'}$, we have
	\begin{equation*}
		\sup_{A\in \G}|\ind\{\omega\in A\} - \ind\{\omega'\in A\} \le 1-\ind\{(\omega,\omega')\in E\} = 0.
	\end{equation*}
	This says that for any $A\in \G$ we have $\omega\in A$ if and only if $\omega'\in A$.
	In other words, $\G\subseteq E^{\ast}$.
	
	Conversely, suppose that $E$ is measurable and that $\G\subseteq E^{\ast}$ and $E\subseteq \G^{\ast}$.
	Note that $E\subseteq \G^{\ast}$ implies that for any $A\in\G$ we have
	\begin{equation*}
		\{(\omega,\omega')\in \Omega\times\Omega:\omega\in A\}\cap E = \{(\omega,\omega')\in \Omega\times\Omega:\omega'\in A\} \cap E.
	\end{equation*}
	Thus, for any $\P,\P'\in \mathcal{P}(\Omega,\F)$ and $\tilde{\P}\in \Pi(\P,\P')$, we can bound:
	\begin{equation*}
		\begin{split}
			\P(A) - \P'(A) = \tilde{\P}(\omega\in A) - \tilde{\P}(\omega'\in A) &=\tilde{\P}(\{\omega\in A\}\setminus E) \\
			&\qquad- \tilde{\P}(\{\omega'\in A\}\setminus E) \\
			&\le \tilde{\P}(\{\omega\in A\}\setminus E).
		\end{split}
	\end{equation*}
	Now take the supremum over $A\in\G$ and use that $\Omega\in \G$, then take the infimum over $\tilde{\P}\in \Pi(\P,\P')$ to conclude.
\end{proof}

As we have seen in Remark~\ref{ref:invariant-salg-ambient} and Example~\ref{ex:ctble-coctble}, the roles of $E$ and $\G$ in our duality theory are not completely symmetric.
While in the bulk of the paper we typically aim to state our results in the most symmetric form possible, we will address, in Subsection~\ref{subsec:equiv} and Subsection~\ref{subsec:salg}, the particularities associated with taking either of $E$ or $\G$ as given.

\begin{corollary}\label{cor:weak-dual}
	For any measurable relation $E$ on $\Omega$, the pair $(E,E^{\ast})$ satisfies weak duality.
	For any $\G\subseteq\F$ such that $\G^{\ast}$ is measurable and $\Omega\in\G$, the pair $(\G^{\ast},\G)$ satisfies weak duality.
\end{corollary}

\begin{proof}
	By part (i) of Lemma~\ref{lem:Galois-basics}, we have $E\subseteq E^{\ast\ast}$ and $\G\subseteq \G^{\ast\ast}$, and we of course also have $\Omega\in E^{\ast}$.
	Thus, the result follows from Proposition~\ref{prop:weak-dual}.
\end{proof}

\subsection{Quasi-Strong Duality}\label{subsec:pseudo-strong}

In this brief subsection we study quasi-strong duality as it relates to strong duality.
We begin with a simple but useful auxiliary result.

\begin{lemma}\label{lem:completion}
	For any $\P,\P'\in \mathcal{P}(\Omega,\F)$, if $\tilde{\M}\in \Pi_{s}(\P,\P')$ is a sub-coupling of $\P$ and $\P'$, then there exists a coupling $\tilde{\P}\in \Pi(\P,\P')$ with $\tilde{\M}\le \tilde{\P}$.
\end{lemma}

\begin{proof}
	By assumption, $\Q := \P-\tilde{\M}\circ \pi^{-1}$ and $\Q' := \P'-\tilde{\M}\circ (\pi')^{-1}$ are sub-probability measures on $(\Omega,\F)$.
	Of course, they have the same total mass, which we denote $\gamma \in [0,1]$.
	More explicitly, we have
	\begin{equation*}
		\gamma = \Q(\Omega) = \P(\Omega)-(\tilde{\M}\circ \pi^{-1})(\Omega) = 1 - \tilde{\M}(\Omega\times\Omega).
	\end{equation*}
	The remainder of the proof proceeds in two cases.
	
	If $\gamma = 0$, we take $\tilde{\P} := \tilde{\M}$.
	To see that it is a probability measure, compute
	\begin{align*}
		0 = \gamma = \Q(\Omega) &= \P(\Omega) - (\tilde{\M}\circ\pi^{-1})(\Omega) \\
		&= 1 - \tilde{\M}(\Omega\times\Omega) \\
		&= 1 - \tilde{\P}(\Omega\times\Omega).
	\end{align*}
	To see that it has the correct marginals, observe that $\tilde{\P}\circ \pi^{-1}\le \P$ with $\tilde{\P}$ a probability measure implies that $\tilde{\P}\circ \pi^{-1}= \P$; likewise for $\tilde{\P}\circ (\pi')^{-1}$.
	
	If $\gamma > 0$, we take
	\begin{equation*}
		\tilde{\P} := \tilde{\M} + \frac{\Q\otimes \Q'}{\gamma}.
	\end{equation*}
	To see that it is a probability measure, compute
	\begin{equation*}
		\tilde{\P}(\Omega\times\Omega) = \tilde{\M}(\Omega\times\Omega) + \frac{\gamma^2}{\gamma} = \tilde{\M}(\Omega\times\Omega) + \gamma = 1.
	\end{equation*}
	To see that it has the correct marginals, compute
	\begin{equation*}
		\tilde{\P}\circ \pi^{-1} = \tilde{\M}\circ \pi^{-1}+\frac{\Q'(\Omega)}{\gamma}\Q = \tilde{\M}\circ \pi^{-1}+\Q = \P,
	\end{equation*}
	and likewise for $\tilde{\P}\circ(\pi')^{-1}$.
\end{proof}

From this we get the following.

\begin{proposition}\label{prop:pseudo-strong-dual}
	On a standard Borel space $(\Omega,\F)$, a pair $(E,\G)$ satisfies strong duality if and only if $E$ is measurable and $(E,\G)$ satisfies quasi-strong duality.
\end{proposition}

\begin{proof}
	Let $E$ be an equivalence relation on $\Omega$ and $\G$ a sub-$\sigma$-algebra of $\F$.
	It is clear that $(E,\G)$ satisfying strong duality implies that it satisfies quasi-strong duality.
	Conversely, suppose that $E$ is measurable and that $(E,\G)$ satisfies quasi-strong duality.
	To see that $(E,\G)$ satisfies strong duality, let us show that an $(E,\G)$-optimal coupling exists for all $\P,\P'\in \mathcal{P}(\Omega,\F)$.
	To do this, consider $\mu := \P|_{\G}$ and $\mu' := \P'|_{\G}$ as probability measures on the measurable space $(\Omega,\G)$, and set $\nu := \mu \wedge \mu'$.
	(Note that $\nu$ is not in general equal to $(\P\wedge \P')|_{\G}$.)
	Next, write $K:\Omega\times\F\to [0,1]$ for a regular conditional distribution of $\P$ with respect to $\G$, and $K':\Omega\times\F\to [0,1]$ for a regular conditional distribution of $\P'$ with respect to $\G$.
	Finally, define sub-probability measures $\M,\M'\in \mathcal{P}_s(\Omega,\F)$ via $\M(A) := \int_{\Omega}K(\omega,A)\, d\nu(\omega)$ and $\M'(A) := \int_{\Omega}K'(\omega,A)\, d\nu(\omega)$ for all $A\in \F$.
	
	By construction we have $\M(\Omega) = \M'(\Omega) = \nu(\Omega)$.
	Moreover, for $A\in\G$ we have $K(\omega,A) = \ind_{A}(\omega)$ holding $\P$-almost surely hence $\mu$-almost surely, and $\nu\ll\mu$ implies that it further holds $\nu$-almost surely; similarly, for $A\in\G$ we have $K'(\omega,A) = \ind_{A}(\omega)$ holding $\nu$-almost surely.
	In particular, for $A\in\G$ we have
	\begin{equation*}
		\M(A) = \int_{\Omega}K(\omega,A)\, d\nu(\omega) = \nu(A) =  \int_{\Omega}K'(\omega,A)\, d\nu(\omega) = \M'(A).
	\end{equation*}
	This means we can apply quasi-strong duality (to a suitable rescaling of $\M$ and $\M'$) to get that there must exist some $\tilde{\M}\in \Pi(\M,\M')$ and some $N\in\F\otimes\F$ with $\tilde{\M}(N) = 0$ and $(\Omega\times\Omega)\setminus E\subseteq N$.
	
	Since $E$ is measurable, this implies
	\begin{equation*}
		\tilde{\M}(\Omega\times\Omega) - \tilde{\M}(E) = \tilde{\M}((\Omega\times\Omega)\setminus E) \le \tilde{\M}(N) = 0
	\end{equation*}
	hence $\tilde{\M}(E) = \tilde{\M}(\Omega\times\Omega) = \nu(\Omega)$.
	In fact, for $A\in \F$ we have
	\begin{align*}
		(\tilde{\M}\circ \pi^{-1})(A) = \M(A) &= \int_{\Omega}K(\omega,A)\, d\nu(\omega) \\
		&\le \int_{\Omega}K(\omega,A)\, d\mu(\omega) \\
		&= \int_{\Omega}K(\omega,A)\, d\P(\omega) = \P(A),
	\end{align*}
	hence $\tilde{\M}\circ \pi^{-1}\le \P$.
	Likewise we have $\tilde{\M}\circ (\pi')^{-1}\le \P'$.
	This gives $\tilde{\M}\in \Pi_s(\P,\P')$, so we can apply Lemma~\ref{lem:completion} to get some $\tilde{\P}\in\Pi(\P,\P')$ with $\tilde{\M}\le \tilde{\P}$.
	
	Finally, let $P,N\in \G$ denote the positive and negative part of the Jordan decomposition of the signed measure $\mu-\mu'$ on $(\Omega,\G)$.
	It is well-known \cite[Corollary~2.9]{Kallenberg} that we have
	\begin{equation*}
		\nu(A)= \mu(A)-\mu(A\cap P)+\mu'(A\cap P)
	\end{equation*}
	for all $A\in \G$, and thus
	\begin{align*}
		\nu(\Omega)  &= 1-\mu(P)+\mu'(P) \\
		&\ge 1- \sup_{A\in \G}|\mu(A)-\mu'(A)| \\
		&= 1- \sup_{A\in \G}|\P(A)-\P'(A)|.
	\end{align*}
	Combining this all, we have shown
	\begin{equation*}
		1-\tilde{\P}(E) \le 1-\tilde{\M}(E) = 1- \nu(\Omega) \le \sup_{A\in \G}|\P(A)-\P'(A)|.
	\end{equation*}
	Since Proposition~\ref{prop:weak-dual} establishes the reverse inequality, we have shown that $\tilde{\P}$ is $(E,\G)$-optimal, hence that $(E,\G)$ satisfies strong duality.
\end{proof}

The value of Proposition~\ref{prop:pseudo-strong-dual} is twofold.
First, it allows us to simplify the proofs of some later results by demonstrating quasi-strong duality in place of strong duality.
Second, it allows us to immediately upgrade many results from the existing literature, which primarily come in the form of quasi-strong duality, into statements of strong duality.
For example, it is known (see \cite[Proposition~3.1]{Georgii} and the remarks thereafter) that if $G$ countable group acting measurably on a standard Borel space $(\Omega,\F)$, then the orbit equivalence relation
\begin{equation*}
	E_{G} := \{(\omega,\omega')\in\Omega\times\Omega: g\cdot\omega = g'\cdot\omega' \textrm{ for some } g\in G \}
\end{equation*}
is measurable, the invariant $\sigma$-algebra
\begin{equation*}
	\invar_{G} := \{A\in\F: g^{-1}(A) = A \textrm{ for all } g\in G\}
\end{equation*}
satisfies $E_G^{\ast} = \invar_G$, and the pair $(E_G,\invar_G)$ satisfies quasi-strong duality.
Thus we conclude:

\begin{theorem}\label{thm:ergodic-duality}
	If $G$ is a countable group acting measurably on a standard Borel space, then $(E_G,\invar_G)$ satisfies strong duality.
\end{theorem}

The result of \cite{Georgii} in fact guarantees quasi-strong duality for a more general collection of semigroups acting measurably on a standard Borel space, and thus we get an analogous statement of strong duality for free.
However, the precise sufficient conditions of \cite{Georgii} are rather cumbersome to state, so we omit them for brevity's sake.
Suffice it to say that quasi-strong duality is known for most orbit equivalence relations, so our results establish strong duality for most measurable orbit equivalence relations.

\subsection{Optimality Conditions}

Before we get to the next subsection in which we prove our main results of interest on strong duality, we address a manifestation of a common principle in mathematical optimization, that identifying a useful duality between classes of optimization problems can lead one to optimality conditions.

Along the way, we will need an important intermediate result.
It establishes that \eqref{eqn:dual} always admits a maximizer, and that the maximization problem is equivalent to a suitable convex relaxation.
Since this result is well-known, we omit its proof.

\begin{lemma}\label{lem:dual-max}
	For any sub-$\sigma$-algebra $\G$ of $\F$ and any $\P,\P'\in \mathcal{P}(\Omega,\F)$, we have
	\begin{equation*}
		\max_{A\in\G}|\P(A)-\P'(A)| = \max_{\substack{f\in b\G \\ 0\le f\le 1}}\left|\int_{\Omega} f\, d\P - \int_{\Omega} f\, d\P'\right|
	\end{equation*}
\end{lemma}

Now we turn to the question of characterizing optimal solutions to the problems \eqref{eqn:primal} and \eqref{eqn:dual}.
We primarily pursue the study minimizers of \eqref{eqn:primal}, which, as we have indicated, might not always exist.
Our main result in this direction is an exact characterization of minimizing couplings.

\begin{proposition}\label{prop:opt-cond}
	For a pair $(E,\G)$ satisfying strong duality and any $\P,\P'\in \mathcal{P}(\Omega,\F)$, a coupling $\tilde{\P}\in\Pi(\P,\P')$ is $(E,\G)$-optimal  if and only if $\tilde{\P}(\cdot \setminus E)\circ\pi^{-1}$ and $\tilde{\P}(\cdot \setminus E)\circ(\pi')^{-1}$ are mutually singular on $(\Omega,\G)$.
\end{proposition}

\begin{proof}
	For the ``if'' direction, we suppose that $\tilde{\P}(\cdot \setminus E)\circ\pi'^{-1}$ and $\tilde{\P}(\cdot \setminus E)\circ(\pi')^{-1}$ are mutually singular on $(\Omega,\G)$, that is, that there exists $A\in \G$ such that we have $(\tilde{\P}(\cdot \setminus E)\circ\pi^{-1})(A) = 1-\tilde{\P}(E)$ and $(\tilde{\P}(\cdot \setminus E)\circ(\pi')^{-1})(A) = 0$.
	By Corollary~\ref{cor:weak-dual}, it now suffices to show that we have $1-\tilde{\P}(E) \le |\P(A)-\P'(A)|$.
	Indeed, we simply bound:
	\begin{align*}
		1-\tilde{\P}(E) &= \tilde{\P}(\{(\omega,\omega')\in\Omega\times\Omega: \ind_{A}(\omega)= 1, \ind_{A}(\omega') = 0\}\setminus E) \\
		&= \int_{\Omega\times\Omega}(\ind_{A}(\omega)-\ind_{A}(\omega'))\,d(\tilde{\P}(\cdot\setminus E))(\omega,\omega') \\
		&= \int_{\Omega\times\Omega}(\ind_{A}(\omega)-\ind_{A}(\omega'))\,d\tilde{\P}(\omega,\omega) = \P(A) - \P'(A) \le |\P(A) - \P'(A)|.
	\end{align*}
	For the ``only if'' direction, we suppose that $\tilde{\P}\in\Pi(\P,\P')$ is $E$-optimal for $\P,\P'$.
	By Lemma~\ref{lem:dual-max}, there is some event $A\in \G$ that is $\G$-optimal for $\P,\P'$.
	Note that $A\in \G$ is equivalent to having $|\ind_{A}(\omega) - \ind_{A}(\omega')| = 1-\ind_{E}(\omega,\omega')$ for all $\omega,\omega'\in\Omega$.
	At the same time, we can compute:
	\begin{align*}
		|\P(A) - \P'(A)| &= \left|\int_{\Omega\times\Omega}(\ind_{A}(\omega)-\ind_{A}(\omega'))\, d\tilde{\P}(\omega,\omega')\right| \\
		&\le \int_{\Omega\times\Omega}|\ind_{A}(\omega)-\ind_{A}(\omega')|\, d\tilde{\P}(\omega,\omega') \\
		&= \int_{\Omega\times\Omega}1-\ind_{E}(\omega,\omega')\, d\tilde{\P}(\omega,\omega') = 1-\tilde{\P}(E)
	\end{align*}
	However, note that $\tilde{\P}$ being $E$-optimal and $A$ being $\G$-optimal combine with strong duality to show that $|\P(A) - \P'(A)| = 1-\tilde{\P}(E)$.
	Consequently, the inequality above is an equality, and we have
	\begin{equation*}
		\left|\int_{\Omega\times\Omega}(\ind_{A}(\omega)-\ind_{A}(\omega'))\, d\tilde{\P}(\omega,\omega')\right| = \int_{\Omega\times\Omega}|\ind_{A}(\omega)-\ind_{A}(\omega')|\, d\tilde{\P}(\omega,\omega').
	\end{equation*}
	This implies that we have either
	\begin{equation*}
		\tilde{\P}(\{(\omega,\omega')\in\Omega\times\Omega: \ind_{A}(\omega)\ge \ind_{A}(\omega') \}) = 1
	\end{equation*}
	or
	\begin{equation*}
		\tilde{\P}(\{(\omega,\omega')\in\Omega\times\Omega: \ind_{A}(\omega)\le\ind_{A}(\omega') \}) = 1.
	\end{equation*}
	In the first case, we use $|\ind_{A}(\omega)-\ind_{A}(\omega')| = 1$ for $(\omega,\omega')\in (\Omega\times\Omega)\setminus E$ to get
	\begin{equation*}
		\tilde{\P}(\{(\omega,\omega')\in\Omega\times\Omega: \ind_{A}(\omega)=1, \ind_{A}(\omega') =0\}\setminus E) = 1-\tilde{\P}(E)
	\end{equation*}
	hence $(\tilde{\P}(\cdot\setminus E)\circ \pi^{-1})(A) =1-\tilde{\P}(E)$ and $(\tilde{\P}(\cdot\setminus E)\circ (\pi')^{-1})(A) = 0$.
	In the second case, the same argument shows
	\begin{equation*}
		\tilde{\P}(\{(\omega,\omega')\in\Omega\times\Omega: \ind_{A}(\omega)=0, \ind_{A}(\omega') =1\}\setminus E) = 1-\tilde{\P}(E)
	\end{equation*}
	hence $(\tilde{\P}(\cdot\setminus E)\circ \pi^{-1})(A) =0$ and $(\tilde{\P}(\cdot\setminus E)\circ (\pi')^{-1})(A) = 1-\tilde{\P}(E)$.
	In either case, we have shown that $\tilde{\P}(\cdot\setminus E)\circ \pi^{-1}$ and $\tilde{\P}(\cdot\setminus E)\circ (\pi')^{-1}$ are mutually singular on $(\Omega,\G)$, as claimed.
\end{proof}

We believe, but are unable to prove, that the previous condition for minimality should be necessary and sufficient in great generality.
For one indication of this, we recall the result \cite[Theorem~1]{BetterTransport}: For a measurable cost function on a standard Borel space, a coupling is minimal for the Monge-Kantorovich optimal transport problem if and only if it is concentrated on a cyclically monotone set.
Thus, our condition is necessary and sufficient if one can show that a set $S\subseteq \Omega\times\Omega$ is $(1-\ind_{E})$-cyclically monotone if and only if there exists an event $A\in \G$ satisfying $S\setminus E \subseteq (A\times \Omega)\cup (\Omega\times (\Omega\setminus A))$.
We believe this is an interesting avenue for future work.

One can certainly also use this duality to characterize maximizers of \eqref{eqn:dual}, but the resulting condition is not so useful:
An event $A\in\G$ is $\G$-optimal if and only if there exists a coupling $\tilde{\P}\in \Pi(\P,\P')$ such that $A$ witnesses the mutual singularity of $\tilde{\P}(\cdot\setminus E)\circ \pi^{-1}$ and $\tilde{\P}(\cdot\setminus E)\circ (\pi')^{-1}$.
We believe it would be interesting to see if this condition can be re-intrepret in a manner which does not reference a quantifier over all possible couplings $\tilde{\P}\in \Pi(\P,\P')$.

\subsection{Strong Duality}\label{subsec:strong}

In this subsection, we finally devote our attention to strong duality.
In particular, we state and prove our main sufficient conditions for strong dualizability for measurable equivalence relations.

\begin{theorem}\label{thm:basic-dual}
	If $(\Omega,\F)$ is a standard Borel space and if a pair $(E,\G)$ satisfies $E\subseteq \G^{\ast}$ and (equivalently, or) $\G\subseteq E^{\ast}$ as well as $E\in \G\otimes \G$, then $(E,\G)$ satisfies strong duality.
\end{theorem}

\begin{proof}
	We see that $\G\subseteq\F$ implies that $E$ is measurable, so, by Proposition~\ref{prop:pseudo-strong-dual}, it suffices to show that $(E,\G)$ satisfies quasi-strong duality.
	Further, note that Proposition~\ref{prop:weak-dual} guarantees that (ii) implies (i) in the definition of quasi-strong duality, so it only remains to show that (i) implies (ii).
	That is, for any $\P,\P'\in\mathcal{P}(\Omega,\F)$ satisfying $\P(A) = \P'(A)$ for all $A\in \G$, we must construct an $E$-successful coupling.
	
	Our construction is as follows.
	First, define $\nu := \P|_{\G}$ as a probability measure on $(\Omega,\G)$, and note by assumption that we also have $\nu = \P'|_{\G}$.
	Second, we use the fact that $(\Omega,\F)$ is standard Borel to get \cite[Corollary~10.4.6]{Bogachev} a regular conditional distribution of $\P$ with respect to $\G$ denoted $K:\Omega\times\F\to [0,1]$ as well as a regular conditional probability of $\P'$ with respect to $\G$ denoted $K':\Omega\times\F\to [0,1]$.
	Next, we define the set-function $\tilde{\P}:\F\times\F\to [0,1]$ via
	\begin{equation*}
		\tilde{\P}(A\times A'):=\int_{\Omega}K(\omega,A)K'(\omega,A')\,d\nu(\omega)
	\end{equation*}
	for $A,A'\in \F$.
	The last step of the construction is to show that $\tilde{\P}$ extends uniquely to a probability measure on $(\Omega\times\Omega,\F\otimes\F)$, which (by a slight abuse of notation) we also denote by $\tilde{\P}$.
	This of course follows from Carath\'eodory's extension theorem  \cite[Theorem~2.5]{Kallenberg} if we can show that $\tilde{\P}$ is countably additive on the semi-ring $\F\times\F$, so suppose that we have $A\times A'= \bigcup_{n\in\N}(A_n\times A_n')$ for $A,A',\{A_n\}_{n\in\N}$, and $\{A_n'\}_{n\in\N}$ in $\F$ such that $\{A_n\times A_n'\}_{n\in\N}$ are disjoint.
	This implies $\ind_A\otimes \ind_{A'} = \sum_{n\in\N}(\ind_{A_n}\otimes \ind_{A_n'})$, so for a fixed $\omega\in\Omega$ we can take the probability of both sides under the product measure $K(\omega,\,\cdot\,)\otimes K'(\omega,\,\cdot\,)$ and we get
	\begin{equation*}
		K(\omega,A)K'(\omega,A') = \sum_{n\in\N}K(\omega, A_n)K'(\omega,A_n').
	\end{equation*}
	Now integrate both sides with respect to $\nu$, and use monotone convergence to get
	\begin{align*}
		\tilde{\P}(A\times A') &= \int_{\Omega}K(\omega,A)K'(\omega,A')\, d\nu(\omega) \\
		&= \int_{\Omega}\sum_{n\in\N}K(\omega, A_n)K'(\omega,A_n')\, d\nu(\omega) \\
		&= \sum_{n\in\N}\int_{\Omega}K(\omega, A_n)K'(\omega,A_n')\, d\nu(\omega) = \sum_{n\in\N}\tilde{\P}(A_n\times A_n').
	\end{align*}
	This is as desired, and thus completes the construction.
	To complete the proof, we need to verify two further properties about $\tilde{\P}$.

	First, we claim that $\tilde{\P}$ is a coupling of $\P$ and $\P'$.
	Indeed, for any $A\in \F$:
	\begin{align*}
		(\tilde{\P}\circ\pi^{-1})(A) &= \tilde{\P}(A\times \Omega) \\
		&=\int_{\Omega} K(\omega,A)K'(\omega,\Omega)\,d\nu(\omega) \\
		&=\int_{\Omega} K(\omega,A)\,d\nu(\omega) =\int_{\Omega} K(\omega,A)\,d\P(\omega) = \P(A),
	\end{align*}
	hence $\tilde{\P}\circ\pi^{-1}= \P$.
	The same calculation shows $\tilde{\P}\circ(\pi')^{-1}= \P'$.
	
	Second, we claim that $\tilde{\P}(E) = 1$.
	To do this, write $f:(\Omega,\G)\to(\Omega\times\Omega,\G\otimes \G)$ for the measurable map $f(\omega) := (\omega,\omega)$, and let us show that we have $\tilde{\P}(S) = \nu(f^{-1}(S))$ for all $S\in\G\otimes \G$.
	Indeed, it is straightforward to show that for all $A,A'\in\G$ we have
	\begin{align*}
	\tilde{\P}(A\times A') &=\int_{\Omega} K(\omega,A)K'(\omega,A')\,d\nu(\omega) \\
	&=\int_{\Omega} \ind_{A}(\omega)\ind_{A'}(\omega)\,d\nu(\omega) = \nu(A\cap A') = \nu(f^{-1}(A\times A')).
	\end{align*}
	In the second equality, we used that $K(\omega,A) = \ind_{A}(\omega)$ holds $\P$-almost surely hence $\nu$-almost surely, and that $K'(\omega,A') = \ind_{A'}(\omega)$ holds $\P'$-almost surely hence $\nu$-almost surely.
	This shows that the probability measures $\tilde{\P}$ and $\nu\circ f^{-1}$ agree on the $\pi$-system $\G\times \G$, so it follows that they agree on $\G\otimes \G$.
	Finally, we use $E\in \G\otimes \G$ to compute:
	\begin{equation*}
		\tilde{\P}(E) = \nu(f^{-1}(E)) \ge \nu(f^{-1}(\Delta)) = \nu(\Omega) = 1.
	\end{equation*}
	This finishes the proof.
\end{proof}

One may be tempted to think that Theorem~\ref{thm:basic-dual} is powerful enough to establish strong duality most cases of interest in probability theory.
However, many examples of $(E,\G)$ of interest do not satisfy $E\in\G\otimes\G$: importantly, the tail equivalence relation $E_0$ and the tail $\sigma$-algebra $\tail$ are such that $(E_0,\tail)$ satisfies strong duality (to see this, combine \cite{Goldstein} with Proposition~\ref{prop:pseudo-strong-dual}), but $E_0\notin \tail\otimes\tail$. 
(The fact that $E_0\notin \tail\otimes\tail$ is closely related to the Glimm-Effros dichotomy \cite[Theorem~6.5]{KechrisCBER}, but it can also be shown directly via elementary considerations.)
Thus, our goal is to widen our sufficient conditions to include $(E_0,\tail)$.
Towards filling this gap, we establish our next main result:

\begin{theorem}\label{thm:ctble-dual}
	On a standard Borel space $(\Omega,\F)$, if equivalence relations $E_1\subseteq E_2\subseteq\cdots$ on $\Omega$ and sub-$\sigma$-algebras $\G_1\supseteq\G_2\supseteq \cdots$ of $\F$ are such that for each $n\in\N$ the pair $(E_n,\G_n)$ satisfies strong duality, then the pair $(\bigcup_{n\in\N}E_n,\bigcap_{n\in\N}\G_n)$ satisfies strong duality.
\end{theorem}

\begin{proof}
	For convenience, write $E:= \bigcup_{n\in\N}E_n$ and $\G :=\bigcap_{n\in\N}\G_n$.
	Since $E_n$ is measurable for all $n\in\N$ it follows that $E$ is measurable.
	Hence, it suffices by Proposition~\ref{prop:pseudo-strong-dual} to show that $(E,\G)$ satisfies quasi-strong duality.
	To see that (ii) implies (i) in the definition of quasi-strong duality, use Proposition~\ref{prop:weak-dual} to see that, for each $n\in\N$, the pair $(E_n,\G_n)$ being strongly hence weakly dual implies $E_n\subseteq \G^{\ast}_n$ for all $n\in\N$.
	Now the antimonotonicity of the Galois correspondence gives $E_n\subseteq \G_n^{\ast}\subseteq \G^{\ast}$ for all $n\in\N$ hence $E\subseteq \G^{\ast}$.
	Thus, Proposition~\ref{prop:weak-dual}, shows that $(E,\G)$ satisfies weak duality, which implies that (ii) implies (i) in the definition of quasi-strong duality.
	To see that (i) implies (ii), we take any $\P,\P'\in \mathcal{P}(\Omega,\F)$ with $\P(A) = \P'(A)$ for all $A\in \G$, and we show how to construct an $(E,\G)$-optimal coupling.
	
	To begin, set $\P_0 := \P$ and $\P_0' := \P'$.
	Then, inductively for $n\in\N$, use the strong duality of $(E_{n},\G_n)$ to let $\tilde{\P}_n$ be an $(E_n,\G_n)$-optimal coupling of $\P_n$ and $\P_n'$, and set $\P_{n+1} := \P_n - \tilde{\P}_{n}(\cdot\cap E_n)\circ \pi^{-1}$ and $\P_{n+1}' := \P_n' - \tilde{\P}_n(\cdot\cap E_n)\circ (\pi')^{-1}$.
	To ensure that this construction is well-defined, we must verify that $\P_{n}$ and $\P_{n}'$ are, for each $n\in\N$, sub-probability measures with the same total mass.
	Indeed, we claim by induction on $n\in\N$ that
	\begin{equation*}
		\P_{n}(\Omega) = \P_{n}'(\Omega) = 1-\sum_{m=0}^{n-1}\tilde{\P}_m(E_m).
	\end{equation*}
	The base case $n=0$ holds because $\P_0$ and $\P_0'$ are probability measures, and the inductive step for $n\in\N$ follows from combining
	\begin{align*}
		\P_{n+1}(\Omega) &= \P_n(\Omega)-(\tilde{\P}_n(\cdot\cap E_n)\circ \pi'^{-1})(\Omega) \\
		&= \P_n(\Omega)-\tilde{\P}_n(E_n) \\
		&= 1-\sum_{m=0}^{n-1}\tilde{\P}_m(E_m)-\tilde{\P}_n(E_n) = 1-\sum_{m=0}^{n}\tilde{\P}_m(E_m)
	\end{align*}
	with the analogous calculation for $\P_{n+1}'(\Omega)$.
	
	Next, fix $n\in\N$, and recall that by construction we have $\P_{n+1} = \P - \sum_{m=0}^{n}\tilde{\P}_m(\cdot \cap E_m)\circ \pi^{-1}$ and $\P_{n+1}' = \P' - \sum_{m=0}^{n}\tilde{\P}_m(\cdot \cap E_m)\circ (\pi')^{-1}$.
	Thus, combining the $(E_{n+1},\G_{n+1})$-optimality of $\tilde{\P}_{n+1}$ with the triangle inequality, we get
	\begin{align*}
		1-&\sum_{m=0}^{n+1}\tilde{\P}_m(E_m) \\
		&= 1-\sum_{m=0}^{n}\tilde{\P}_m(E_m) - \tilde{\P}_{n+1}(E_{n+1}) \\
		&= \sup_{A\in \G_{n+1}}|\P_{n+1}(A)-\P_{n+1}'(A)| \\
		&\le \sup_{A\in \G_{n+1}}|\P(A) - \P'(A)| \\
		&+\sum_{m=0}^{n}\sup_{A\in \G_{n+1}}|(\tilde{\P}_{m}(\cdot\cap E_m)\circ \pi^{-1})(A) - (\tilde{\P}_m(\cdot\cap E_m)\circ (\pi')^{-1})(A)|.
	\end{align*}
	We claim that the sum in the last line above is equal to zero.
	In fact, we claim that all summands are equal to zero, in that for all $m=0,1,\ldots,n$ we have
	\begin{equation}\label{eqn:ctble-dual-1}
		\sup_{A\in \G_{n+1}}|(\tilde{\P}_{m}(\cdot\cap E_m)\circ \pi^{-1})(A) - (\tilde{\P}_m(\cdot\cap E_m)\circ (\pi')^{-1})(A)| = 0.
	\end{equation}
	Since
	\begin{equation*}
		(\tilde{\P}_{m}(\cdot\cap E_m)\circ \pi^{-1})(\Omega) = (\tilde{\P}_m(\cdot\cap E_m)\circ (\pi')^{-1})(\Omega) = \tilde{\P}_m(E_m),
	\end{equation*}
	it follows by normalizing that \eqref{eqn:ctble-dual-1} holds if and only if we have
	\begin{equation}\label{eqn:ctble-dual-2}
		\sup_{A\in \G_{n+1}}|(\tilde{\P}_{m}(\cdot\,|\, E_m)\circ \pi^{-1})(A) - (\tilde{\P}_m(\cdot\,|\, E_m)\circ (\pi')^{-1})(A)| = 0.
	\end{equation}
	By the strong duality of $(E_{n+1},\G_{n+1})$ we know that \eqref{eqn:ctble-dual-2} holds if and only if there exists an $E_{n+1}$-successful coupling of $\tilde{\P}_m(\cdot\,|\, E_m)\circ\pi^{-1}$ and $\tilde{\P}_m(\cdot\,|\, E_m)\circ(\pi')^{-1}$.
	Of course, the coupling $\tilde{\P}_m(\cdot\,|\,E_m)$ is exactly what is needed, since $E_m\subseteq E_{n+1}$ implies $\tilde{\P}_m(E_{n+1}\,|\,E_m) = 1$.
	Thus, we have shown
	\begin{equation}\label{eqn:ctble-dual-3}
		1-\sum_{m=0}^{n+1}\tilde{\P}_m(E_m) \le \sup_{A\in \G_{n+1}}|\P(A) - \P'(A)|
	\end{equation}
	for all $n\in\N$.
	
	Now let us get for each $n\in\N$ some $A_n\in \G_n$ with
	\begin{equation*}
		|\P(A_n) - \P'(A_n)|\ge \sup_{A\in \G_{n}}|\P(A) - \P'(A)|-\frac{1}{2^n}
	\end{equation*}
	Now consider the Hilbert space $L^2(\Omega,\F,\frac{1}{2}(\P+\P'))$, in which $\{\ind_{A_n}\}_{n\in\N}$ form a norm-bounded sequence.
	By the Banach-Alaoglu theorem, there exists a subsequence $\{n_j\}_{j\in\N}$ and some $f\in L^2(\Omega,\F,\frac{1}{2}(\P+\P'))$ with $\ind_{A_{n_j}}\to f$ weakly.
	Since $L^2(\Omega,\G_n,\frac{1}{2}(\P+\P'))$ is a strongly closed subspace of $L^2(\Omega,\F,\frac{1}{2}(\P+\P'))$, it follows that it is also weakly closed.
	This implies that $f\in L^2(\Omega,\G_n,\frac{1}{2}(\P+\P'))$ for all $n\in\N$, hence $f\in L^2(\Omega,\G,\frac{1}{2}(\P+\P'))$.
	Also, we have
	\begin{equation*}
		0 \le \lim_{j\to\infty}\int_{\Omega}\ind_{A_{n_j}}\ind_{\{f\le 0\}}\,d\left(\frac{\P+\P'}{2}\right) = \int_{\Omega}f\ind_{\{f\le 0\}}\,d\left(\frac{\P+\P'}{2}\right) \le 0
	\end{equation*}
	which shows that $f\ge 0$ holds $\P$- and $\P'$-almost surely; a similar argument shows that $f\le 1$ holds $\P$- and $\P'$-almost surely.
	Putting this all together, we conclude that there exists a function $g:\Omega\to\R$ which is $\G$-measurable and satisfies $0\le g\le 1$ such that $f=g$ both $\P$- and $\P'$-almost surely.
	Consequently, Lemma~\ref{lem:dual-max} gives:
	\begin{align*}
		\liminf_{n\to\infty}\sup_{A\in \G_{n}}|\P(A) - \P'(A)| &\le \lim_{j\to\infty}|\P(A_{n_j}) - \P'(A_{n_j})| \\
		&= \left|\int_{\Omega}f\,d\P - \int_{\Omega}f\, d\P'\right| \\
		&= \left|\int_{\Omega}g\,d\P - \int_{\Omega}g\, d\P'\right| \\
		&\le \sup_{A\in \G} |\P(A) - \P'(A)|.
	\end{align*}
	Therefore, we conclude
	\begin{equation}\label{eqn:ctble-dual-4}
		1-\sum_{n\in\N}\tilde{\P}_n(E_n) \le \sup_{A\in \G}|\P(A) - \P'(A)|.
	\end{equation}
	by taking $n\to\infty$ in \eqref{eqn:ctble-dual-3}.
	
	Now we have all of the ingredients to construct our coupling.
	First, set
	\begin{equation*}
		\tilde{\M} := \sum_{n\in\N}\tilde{\P}_n(\cdot\cap E_n),
	\end{equation*}
	which is evidently a sub-probability measure on $(\Omega\times\Omega,\F\otimes\F)$.
	Next, we claim that $\tilde{\M}$ is a sub-coupling of $\P$ and $\P'$.
	Indeed, for any $A\in \F$ and $n\in\N$:
	\begin{align*}
		\sum_{m=0}^{n}(\tilde{\P}_m(\cdot \cap E_m)\circ \pi^{-1})(A) &= \sum_{m=0}^{n}(\P_{m+1}(A)-\P_{m+1}(A)) \\
		&= \P(A)-\P_{n+1}(A) \\
		&\le \P(A).
	\end{align*}
	Thus, taking $n\to\infty$ we get
	\begin{equation*}
		(\tilde{\M}\circ \pi^{-1})(A) = \sum_{n\in\N}(\tilde{\P}_n(\cdot \cap E_n)\circ \pi^{-1})(A) \le \P(A)
	\end{equation*}
	as claimed.
	We also get $\tilde{\M}\circ \pi^{-1}\le \P$ by the same calculation.
	Lastly, we apply Lemma~\ref{lem:completion} to get some $\tilde{\P}\in \Pi(\P,\P')$ with $\tilde{\M}\le \tilde{\P}$.
	
	It only remains to show that $\tilde{\P}$ is $(E,\G)$-optimal.
	Indeed, note that for all $n\in\N$:
	\begin{equation*}
		\tilde{\M}(E_{n+1})  \ge \sum_{m=0}^{n+1}\tilde{\P}_m(E_m\cap E_{n+1}) = \sum_{m=0}^{n+1}\tilde{\P}_m(E_m).
	\end{equation*}
	Thus, taking $n\to\infty$ gives
	\begin{equation*}
		\tilde{\M}(E) \ge \sum_{n\in\N}\tilde{\P}_n(E_n).
	\end{equation*}
	Finally, by applying $\tilde{\M}\le \tilde{\P}$ and \eqref{eqn:ctble-dual-4}, we find
	\begin{equation*}
		1-\tilde{\P}(E) \le 1-\tilde{\M}(E) \le 1- \sum_{n\in\N}\tilde{\P}_n(E_n) \le \sup_{A\in \G}|\P(A) - \P'(A)|.
	\end{equation*}
	This shows that $\tilde{\P}$ is $(E,\G)$-optimal, so the result is proved.
\end{proof}

\subsection{Given an equivalence relation}\label{subsec:equiv}

In this subsection we consider the setting that $E$ is a given equivalence relation on $\Omega$ and that one wants to find a suitable sub-$\sigma$-algebra $\G$ of $\F$ such that $(E,\G)$ satisfies strong duality.
Indeed, we show that this is possible in great generality, and such results will follow readily by specializing the abstract results of the previous subsection.

The first result shows that $E^{\ast}$ is a canonical choice of $\G$:

\begin{proposition}\label{prop:strong-dual}
	For $E$ a relation on $\Omega$ and $\G$ a subset of $\F$ with $\Omega\in\G$, if the pair $(E,\G)$ satisfies strong duality, then $(E,E^{\ast})$ satisfies strong duality.
\end{proposition}

\begin{proof}
	By Proposition~\ref{prop:weak-dual}, the weak duality of $(E,\G)$ implies $\G\subseteq E^{\ast}$.
	Thus, for any $\P,\P'\in \mathcal{P}(\Omega,\F)$ we can use the strong duality of $(E,\G)$ to bound:
	\begin{equation*}
		\sup_{A\in E^{\ast}}|\P(A)-\P'(A)| \ge \sup_{A\in\G}|\P(A)-\P'(A)| = \min_{\tilde{\P}\in \Pi(\P,\P')}(1-\tilde{\P}(E))
	\end{equation*}
	At the same time, Corollary~\ref{cor:weak-dual} gives the bound:
	\begin{equation*}
		\sup_{A\in E^{\ast}}|\P(A)-\P'(A)| \le \inf_{\tilde{\P}\in \Pi(\P,\P')}(1-\tilde{\P}(E)).
	\end{equation*}
	Combining these completes the proof.
\end{proof}

Motivated by this result, a measurable equivalence relation $E$ on $\Omega$ has $(E,E^{\ast})$ satisfying strong duality if and only if there exists a sub-$\sigma$-algebra $\G$ of $\F$ such that $(E,\G)$ satisfies strong duality; in this case we say that $E$ is \textit{strongly dualizable}.
Let us also say that $E$ is \textit{weakly dualizable} if $(E,E^{\ast})$ satisfies weak duality, and that it is \textit{quasi-strongly dualizable} if $(E,E^{\ast})$ satisfies quasi-strong duality.

Next, we require some basic definitions.
Let us say that an equivalence relation $E$ on a standard Borel space $(\Omega,\F)$ is \textit{smooth} if there exists a standard Borel space $(X,\mathcal{X})$ and some measurable map $\phi:(\Omega,\F)\to (X,\mathcal{X})$ such that $\omega,\omega'\in \Omega$ have $(\omega,\omega')\in E$ if and only if $\phi(\omega) = \phi(\omega')$.
Roughly speaking, a smooth equivalence relation is one that is ``generated by'' equality of Polish-space valued random variables.
Then, we have the following equivalent characterizations of smoothness:

\begin{lemma}\label{lem:smooth-basic}
	For an equivalence relation $E$ on $\Omega$, the following are equivalent:
	\begin{enumerate}
		\item[(i)] $E$ is smooth.
		\item[(ii)] $E^{\ast}$ is countably-generated.
		\item[(iii)] $E\in E^{\ast}\otimes E^{\ast}$.
	\end{enumerate}
\end{lemma}

\begin{remark}
	Although it is not in our notation, the properties (i), (ii), and (iii) all depend on the ambient $\sigma$-algebra $\F$ with which $\Omega$ is endowed. (cf. Remark~\ref{ref:invariant-salg-ambient})
\end{remark}

\begin{proof}
	It is classical that (i) and (ii) are equivalent (see \cite[Exercise~5.1.10]{Srivastava}), so it suffices to prove that (i) and (iii) are equivalent.
	To show (i) implies (iii), suppose $E$ is smooth so there exists a standard Borel space $(X,\mathcal{X})$ and a measurable map $\phi:\Omega\to X$ such that $\omega,\omega'\in \Omega$ have $(\omega,\omega')\in E$ if and only if $\phi(\omega) =\phi(\omega')$.
	It readily follows that $\phi:(\Omega,E^{\ast})\to (X,\mathcal{X})$ is measurable, hence also that $f:(\Omega\times\Omega,E^{\ast}\otimes E^{\ast})\to (X\times X,\mathcal{X}\otimes\mathcal{X})$ defined via $f(\omega,\omega'):=(\phi(\omega),\phi(\omega'))$ is measurable.
	Finally, observe that we have $\Delta\in \mathcal{X}\otimes\mathcal{X}$ since $(X,\mathcal{X})$ is standard Borel, hence $E =f^{-1}(\Delta)\in E^{\ast}\otimes E^{\ast}$.
	Thus,  (iii) holds.
	
	It requires a bit more work to show (iii) implies (i), so suppose $E\in E^{\ast}\otimes E^{\ast}$.
	It is classical that $\Sigma_1 := \{B\in E^{\ast}\otimes E^{\ast}: \textrm{there exist } A_1,A_2,,\ldots \in E^{\ast}\times E^{\ast} \textrm{ with } B\in \sigma(A_m\times A_n: m,n\in\N)\}$ is a $\sigma$-algebra containing $E^{\ast}\times E^{\ast}$, hence $E^{\ast}\otimes E^{\ast}\subseteq \Sigma_1$.
	In particuar, there exist $A_1,A_2,\ldots \in E^{\ast}\times E^{\ast}$ with $E\in \sigma(A_m\times A_n: m,n\in\N)$.
	
	Next, we aim to show that the equivalence classes $[\omega]_{E} := \{\omega'\in \Omega: (\omega,\omega')\in E\}$ are in $\sigma(A_n: n\in\N)$ for all $\omega\in \Omega$.
	To do this, take arbitrary $\omega\in\Omega$ and define $\Sigma_2(\omega) := \{B\in E^{\ast}\otimes E^{\ast}: ([\omega]_E\times [\omega]_E)\cap B\in \sigma(A_m\times A_n: m,n\in\N)\}$, which is easily seen to be a $\sigma$-algebra.
	Moreover, observe that for all $m,n\in\N$ we have
	\begin{equation*}
		([\omega]_E\times [\omega]_E)\cap (A_m\times A_n) = \begin{cases}
			A_m\times A_n, &\textrm{ if } \omega\in A_m\cap A_n, \\
			\emptyset, &\textrm{ if } \omega\notin A_m\cap A_n,
		\end{cases}
	\end{equation*}
	since $A_m,A_n\in E^{\ast}$.
	This implies $\sigma(A_n\times A_m: m,n\in\N)\subseteq \Sigma_2(\omega)$, hence $E\in \Sigma_2(\omega)$.
	Consequently, $[\omega]_E\times [\omega]_E = ([\omega]_E\times [\omega]_E)\cap E\in\sigma(A_m\times A_n: m,n\in\N)\subseteq \sigma(A_n: n\in\N)\otimes \sigma(A_n: n\in\N)$.
	Now Fubini's theorem gives $[\omega]_E\in\sigma(A_n: n\in\N)$, as claimed.
	
	Moving on, we claim that,  for all $\omega,\omega'\in\Omega$ with $(\omega,\omega')\notin E$, there exists $n\in\N$ such that we have either $[\omega]_{E}\subseteq A_n$ and $[\omega']_{E}\cap A_n = \emptyset$ or $[\omega]_{E}\subseteq \Omega\setminus A_n$ and $[\omega']_{E}\subseteq A_n$.
	If this is not true, then, recalling $\{A_n\}_{n\in\N}\subseteq E^{\ast}$, there must exist $\omega,\omega'\in\Omega$ with $(\omega,\omega')\notin E$ such that for all $n\in\N$ we have $[\omega]_{E},[\omega']_{E}\subseteq A_n$ or $[\omega]_{E},[\omega']_{E} \subseteq \Omega\setminus A_n$.
	But $\Sigma_3(\omega,\omega') := \{A\in E^{\ast}: [\omega]_E,[\omega']_E\subseteq A \textrm{ or } [\omega]_E,[\omega']_E\subseteq \Omega\setminus A\}$ is a $\sigma$-algebra, so $\sigma(A_n: n\in\N)\subseteq \Sigma_3(\omega,\omega')$.
	This contradicts the conclusion of the previous paragraph.
	
	Finally, we define the function $\phi:(\Omega,\F)\to (\R,\mathcal{B}(\R))$ via the summation $\phi(\omega) := \sum_{n\in\N}3^{-n}\ind_{A_n}(\omega)$ for all $\omega\in\Omega$, which is clearly measurable.
	Also, the previous paragraph shows that $\omega,\omega'\in\Omega$ have $(\omega,\omega')\in E$ if and only if $\phi(\omega) = \phi(\omega')$.
	Therefore, $E$ is smooth, so (i) holds.
	This finishes the proof.
\end{proof}

This characterization leads us to our main result:

\begin{corollary}\label{cor:main-eq}
	On a standard Borel space, any equivalence relation that can be written as a countable increasing union of smooth equivalence relations is strongly dualizable.
\end{corollary}

\begin{proof}
	Suppose that $E_1\subseteq E_2\subseteq\cdots$ are smooth equivalence relations and $E=\bigcup_{n\in\N}E_n$.
	For each $n\in\N$, we have $E_n\in E_n^{\ast}\otimes E_n^{\ast}$ by Lemma~\ref{lem:smooth-basic} and $E_n\subseteq E_n^{\ast\ast}$ by part (i) of Lemma~\ref{lem:Galois-basics}; thus, Theorem~\ref{thm:basic-dual} implies that $(E_n,E_n^{\ast})$ satisfies strong duality for all $n\in\N$.
	Consequenently, Theorem~\ref{thm:ctble-dual} implies that $(E,\bigcap_{n\in\N}E_n^{\ast})$ satisfies strong duality, hence $E$ is strongly dualizable.
\end{proof}

In descriptive set theory, an equivalence relation that can be written as a countable increasing union of smooth equivalence relations is called a \textit{hypersmooth} equivalence relation.

We also prove the following result, which is interesting from the point of view of descriptive set theory,
Recall that a \textit{Borel equivalence relation} is a measurable equivalence relation on a standard Borel space, and that a \textit{countable Borel equivalence relation} is a Borel equivalence relation whose equivalences classes are all countable.

\begin{corollary}\label{cor:CBER-dual}
	On a standard Borel space, all countable Borel equivalence relations are strongly dualizable.
\end{corollary}

\begin{proof}
	The Feldman-Moore theorem \cite[Theorem~2.3]{KechrisCBER} states that, for any countable Borel equivalence relation $E$ on a standard Borel space $(\Omega,\F)$, there exists a countable group $G$ acting measurably on $\Omega$ such that $E=E_{G}$.
	Therefore, Theorem~\ref{thm:ergodic-duality} guarantees that $E$ is strongly dualizable.
\end{proof}

\subsection{Given a sub-$\sigma$-algebra}\label{subsec:salg}

In this subsection we consider the setting that $\G$ is a given sub-$\sigma$-algebra of $\F$ and that one wants to find a suitable equivalence relation $E$ on $\Omega$ such that $(E,\G)$ satisfies strong duality.
As before, we show that this is possible in great generality, as a consequence of our general results.

We begin, however, with a negative result which shows that the task in this setting can be impossible:

\begin{example}\label{ex:asymmetry}
	Take $\Omega := [0,1]$ with $\F$ its Borel $\sigma$-algebra, and let $\G$ be the $\sigma$-algebra of all subsets of $\Omega$ which are countable or whose complements are countable.
	Now suppose that $E$ is an equivalence relation on $\Omega$ such that the pair $(E,\G)$ satisfies strong duality.
	By Proposition~\ref{prop:weak-dual}, we have $\G^{\ast}\supseteq E$.
	Since $\G$ separates points, we also have $\G^{\ast} = \Delta$.
	Thus, these considerations give $E \subseteq \G^{\ast} = \Delta$.
	Since $E$ is an equivalence relation, it also satisfies $E\supseteq \Delta$, hence we have $E = \Delta$.
	Then, let $\lambda$ denote the Lebesgue measure on $(\Omega,\F)$, and define probability measures $\P$ and $\P'$ via
	\begin{equation*}
		\frac{d\P}{d\lambda}(x) = 2x \qquad \textrm{ and } \qquad \frac{d\P'}{d\lambda}(x) = 2(1-x)
	\end{equation*}
	for $x\in \Omega$.
	It is easy to see that we have
	\begin{equation*}
		\sup_{A\in \G}|\P(A) - \P'(A)| = 0.
	\end{equation*}
	However, since we have $1-\ind_{\Delta}(x,x')\ge x-x'$ for all $x,x'\in\Omega$ it also follows that
	\begin{align*}
		\min_{\tilde{\P}\in\Pi(\P,\P')}(1-\tilde{\P}(\Delta)) \ge \min_{\tilde{\P}}\Bigg(&\int_{\Omega}xd\P(x) - \int_{\Omega}x'd\P'(x')\Bigg) \\
		=&\int_{0}^{1}2x^2d\lambda(x) - \int_{0}^{1}2x'(1-x')d\lambda(x') = \frac{2}{3}- \frac{1}{3} = \frac{1}{3} > 0.
	\end{align*}
	Since this contradicts the assumption that $(E,\G)$ satisfies strong duality, there can be no equivalence relation $E$ on $\Omega$ for which the pair $(E,\G)$ satisfies strong duality.
\end{example}

Despite this difficulty, we will develop some wide sufficient conditions for a positive answer, and these will be general enough to cover most cases of interest to probabilists.

Next we show that $\G^{\ast}$ is a (somewhat) canonical choice of $E$, although observe that it can be hard to directly check for the measurability of $\G^{\ast}$ given $\G$:

\begin{proposition}\label{prop:strong-dual-reverse}
	For $\G$ a sub-$\sigma$-algebra of $\F$ and $E$ an equivalence relation on $\Omega$, if the pair $(E,\G)$ satisfies strong duality and $\G^{\ast}$ is measurable, then $(\G^{\ast},\G)$ satisfies strong duality.
\end{proposition}

\begin{proof}
	By part (i) of Lemma~\ref{lem:Galois-basics} we have $\G\subseteq \G^{\ast\ast}$, so Proposition~\ref{prop:weak-dual} implies that $(\G^{\ast},\G)$ satisfies weak duality.
	To see that $(\G^{\ast},\G)$ in fact satisfies strong duality, take arbitrary $\P,\P'\in \mathcal{P}(\Omega,\F)$, and use the strong duality of $(E,\G)$ to get $\tilde{\P}\in \Pi(\P,\P')$ with
	\begin{equation*}
		\sup_{A\in\G}|\P(A)-\P'(A)| = 1-\tilde{\P}(E).
	\end{equation*}
	Now note that $(E,\G)$ satisfying strong duality implies that it satisfies weak duality, whence $E\subseteq \G^{\ast}$ by Proposition~\ref{prop:weak-dual}.
	Consequently, $1-\tilde{\P}(E) \ge 1-\tilde{\P}(\G^{\ast})$, and this completes the proof.
\end{proof}

As before, let us say that a sub-$\sigma$-algebra $\G$ of $\F$ is \textit{strongly dualizable} if $(\G^{\ast},\G)$ satisfies strong duality.
In contrast to the previous subsection, this need not be equivalent to the existence of some measurable equivalence relation $E$ on $\Omega$ such that $(E,\G)$ satisfies strong duality, since $\G^{\ast}$ can fail to be measurable.
Let us also say that $\G$ is \textit{weakly dualizable} or \textit{quasi-strongly dualizable} if $(\G^{\ast},\G)$ satisfies weak duality or quasi-strong duality, respectively.

Now we turn to another important result.

\begin{lemma}\label{lem:cg-characterization}
	If a sub-$\sigma$-algebra $\G$ of $\F$ is countably-generated, then $\G^{\ast}\in\G\otimes\G$.
\end{lemma}

\begin{proof}
	Suppose that there exists a sequence $\{A_n\}_{n\in\N}$ in $\G$ satisfying $\G = \sigma(A_n: n\in\N)$.
	We claim in this case that
	\begin{equation*}
		\G^{\ast} = \left\{(\omega,\omega')\in \Omega\times\Omega: \forall n\in\N(\omega\in A_n\Leftrightarrow \omega'\in A_n)\right\}.
	\end{equation*}
	It is obvious that the right side contains $\G^{\ast}$, so it only remains to show that the right side is contained in $\G^{\ast}$.
	Indeed, suppose that $(\omega,\omega')$ is an element of the right side.
	Then $\Sigma(\omega,\omega'):=\{A\in\G: \omega\in A\Leftrightarrow \omega'\in A\}$ is a $\sigma$-algebra containing $\{A_n\}_{n\in\N}$, so it must contain $\sigma(A_n: n\in\N)=\G$, and this yields $(\omega,\omega')\in \G^{\ast}$.
	Thus, we can write the expression above as
	\begin{align*}
		\G^{\ast} &= \bigcap_{n\in\N}(((\Omega\setminus A_n)\times \Omega)\cup(\Omega\times A_n)) \\
		&\qquad\qquad \cap \bigcap_{n\in\N}((\Omega\times (\Omega\setminus A_n))\cup(A_n\times \Omega)).
	\end{align*}
	Since each set on the right side is in $\G$, it follows that $\G^{\ast}\in\G\otimes\G$.
\end{proof}

Finally, we come to the main result of this subsection:

\begin{corollary}\label{cor:ctble-dual-salg}
	On a standard Borel space, any sub-$\sigma$-algebra that can be written as a countable decreasing intersection of countably-generated sub-$\sigma$-algebras is strongly dualizable.
\end{corollary}

\begin{proof}
	Suppose that $\G_1\supseteq \G_2\supseteq\cdots$ are countably-generated sub-$\sigma$-algebras and $\G=\bigcap_{n\in\N}\G_n$.
	For each $n\in\N$, we have $\G_n^{\ast}\in \G_n\otimes \G_n$ by Lemma~\ref{lem:cg-characterization} and $\G_n\subseteq \G_n^{\ast\ast}$ by part (i) of Lemma~\ref{lem:Galois-basics}; thus, Theorem~\ref{thm:basic-dual} implies that $(\G_n^{\ast},\G_n)$ satisfies strong duality for all $n\in\N$.
	Consequently, Theorem~\ref{thm:ctble-dual} implies that $(\bigcup_{n\in\N}\G_n^{\ast},\G)$ satisfies strong duality.
	Thus, it only remains to show that $\bigcup_{n\in\N}\G_n^{\ast} = \G^{\ast}$.
	To do this, note that we readily have $\bigcup_{n\in\N}\G_n^{\ast}\subseteq (\bigcap_{n\in\N}\G_n)^{\ast}$ by the antimonotonicity of the Galois correspondence.
	For the converse, suppose that $(\omega,\omega')$ is not in the left side.
	Then for each $n\in\N$ there exists $A_n\in \G_n$ with $\omega\in A_n$ and $\omega'\notin A_n$.
	It follows that the set $A:=\bigcap_{m\in\N}\bigcup_{n\ge m} A_n$ has $A\in \bigcap_{n\in\N}\G_n$ and $\omega\in A$ and $\omega'\notin A$, whence $(\omega,\omega')$ is not in the right side.
	This shows $\bigcup_{n\in\N}\G_n^{\ast}= (\bigcap_{n\in\N}\G_n)^{\ast} = \G^{\ast}$, so $\G$ is strongly dualizable.
\end{proof}

Some authors refer to a sub-$\sigma$-algebra that can be written as a countable decreasing intersection of countably-genreated sub-$\sigma$-algebras a \textit{tail $\sigma$-algebra}.

Let us conclude by giving an equivalent definition of $\G^{\ast}$ which cane be desirable in some situations.
We write $\G(\omega):= \cap\{A\in\G: \omega\in A\}$ for $\omega\in\Omega$ as the intersection of all events in $\G$ which contain $\omega\in\Omega$; this is called the \textit{atom} of $\G$ at $\omega$, and we write $\atom(\G) = \{\G(\omega): \omega\in\Omega\}$ for the collection of all atoms of $\G$.
Then it is easy to show that we have
\begin{equation*}
	\G^{\ast} := \bigcup_{H\in \atom(\G)}H\times H.
\end{equation*}
Note that, in general, atoms of $\G$ need not be $\G$-measurable and, as always, that $\G^{\ast}$ need not be $(\F\otimes\F)$-measurable.

The difficulties of this subsection, compared the previous subsection, are largely due to the fact that $\G^{\ast}$ need not be measurable.
Thus, in light of Proposition~\ref{prop:pseudo-strong-dual}, the study of quasi-strong duality may be more natural, compared to the study of strong duality, for sub-$\sigma$-algebras $\G$.

%%%%%%%%%%%%%%%%%%%%%%%%%%%%%%%%%%%%%%%%%%%%%%%%%%%%%%%%%%%%%%%%%%%
%%                                                               %%
%% No need for \maketitle.                                       %%
%%                                                               %%
%%%%%%%%%%%%%%%%%%%%%%%%%%%%%%%%%%%%%%%%%%%%%%%%%%%%%%%%%%%%%%%%%%%

%%%%%%%%%%%%%%%%%%%%%%%%%%%%%%%%%%%%%%%%%%%%%%%%%%%%%%%%%%%%%%%%%%%
%%                                                               %%
%% Please replace what follows by the body of your article       %%
%% (up to the bibliography):                                     %%
%%                                                               %%
%%%%%%%%%%%%%%%%%%%%%%%%%%%%%%%%%%%%%%%%%%%%%%%%%%%%%%%%%%%%%%%%%%%

%\tableofcontents

%%%%%%%%%%%%%%%%%%%%%%%%%%%%%%%%%%%%%%%%%%%%%%%%%%%%%%%%%%%%%%%%%%%
%%                                                               %%
%% Use the two commands below for producing your bibliography    %%
%% with bibtex, then comment again the commands and include the  %%
%% content of the .bbl file in this file below the commands.     %%
%%                                                               %%
%%%%%%%%%%%%%%%%%%%%%%%%%%%%%%%%%%%%%%%%%%%%%%%%%%%%%%%%%%%%%%%%%%%

%\bibliographystyle{amsplain}
%\bibliography{yourbibfilename}

% add below the content of your .bbl file produced by bibtex.

\providecommand{\bysame}{\leavevmode\hbox to3em{\hrulefill}\thinspace}
\providecommand{\MR}{\relax\ifhmode\unskip\space\fi MR }
\providecommand{\MRhref}[2]{%
	\href{http://www.ams.org/mathscinet-getitem?mr=#1}{#2}
}
\providecommand{\href}[2]{#2}

%%%%%%%%%%%%%%%%%%%%%%%%%%%%%%%%%%%%%%%%%%%%%%%%%%%%%%%%%%%%%%%%%%%
%%                                                               %%
%% You may add acknowledgments (optional).                       %%
%%                                                               %%
%%%%%%%%%%%%%%%%%%%%%%%%%%%%%%%%%%%%%%%%%%%%%%%%%%%%%%%%%%%%%%%%%%%

\subsection*{Acknowledgments}

This material is based upon work for which the author was supported by the National Science Foundation Graduate Research Fellowship under Grant No. DGE 1752814
We thank Steve Evans for many useful conversations, and we thank Jonathan Niles-Weed and Marcel Nutz for their insights about the relationship between strong duality and Kantorovich duality.
We also thank the anonymous reviewer for many suggestions that improved the quality of this paper.

%%%%%%%%%%%%%%%%%%%%%%%%%%%%%%%%%%%%%%%%%%%%%%%%%%%%%%%%%%%%%%%%%%%
%%                                                               %%
%% You have reached the end of your document.                    %%
%%                                                               %%
%%%%%%%%%%%%%%%%%%%%%%%%%%%%%%%%%%%%%%%%%%%%%%%%%%%%%%%%%%%%%%%%%%%

\end{document}